\documentclass{amsart} 

\usepackage{amsmath,amssymb,amsthm} 
\usepackage{graphicx} 
\graphicspath{{./figures/1/}{./figures/}}
\usepackage{subfig}
\usepackage[arrow, matrix, curve]{xy} 
\usepackage{mathtools}
\usepackage{color}
\usepackage[hidelinks]{hyperref} 

\newcommand{\Z}{\mathbb{Z}}
\newcommand{\Q}{\mathbb{Q}}
\newcommand{\N}{\mathbb{N}}

\newcommand{\cS}{\mathcal{S}}

\DeclareMathOperator{\lbracket}{[\![}
\DeclareMathOperator{\rbracket}{]\!]}

\usepackage{amsthm}
\usepackage{thmtools}
\declaretheoremstyle[notefont=\bfseries,notebraces={}{},%
    headpunct={},postheadspace=1em]{mystyle}
\declaretheorem[style=mystyle,numbered=no,name=Theorem]{thm-hand}

\newtheoremstyle{thm}{}{}{\itshape}{}{\bfseries}{}{ }{} 
\newtheoremstyle{definition}{}{}{}{}{\bfseries}{}{ }{} 

\theoremstyle{thm}
\newtheorem{Theorem}{Theorem}[section]
\newtheorem{theorem}[Theorem]{Theorem}
\newtheorem{lemma}[Theorem]{Lemma}
\newtheorem{proposition}[Theorem]{Proposition}
\newtheorem{corollary}[Theorem]{Corollary}

\newtheorem*{theorem*}{Theorem}
\newtheorem{conjecture}[Theorem]{Conjecture}

\newtheorem*{claim}{Claim:}

\theoremstyle{definition}
\newtheorem{definition}[Theorem]{Definition}
\newtheorem{rem}[Theorem]{Remark}
\newtheorem{ex}[Theorem]{Example}

\begingroup\expandafter\expandafter\expandafter\endgroup
\expandafter\ifx\csname pdfsuppresswarningpagegroup\endcsname\relax
\else
\fi

\definecolor{amaranth}{rgb}{0.9, 0.17, 0.31} 
\definecolor{carrotorange}{rgb}{0.80, 0.5, 0.01} 
\definecolor{citrine}{rgb}{0.89, 0.82, 0.04} 
\definecolor{dartmouthgreen}{rgb}{0.05, 0.5, 0.06} 
\definecolor{ballblue}{rgb}{0.13, 0.67, 0.8} 
\definecolor{ceruleanblue}{rgb}{0.16, 0.32, 0.75} 
\definecolor{amethyst}{rgb}{0.6, 0.4, 0.8} 
\definecolor{amber}{rgb}{1.0, 0.75, 0.0} 
\definecolor{burlywood}{rgb}{0.87, 0.72, 0.53} 

\begin{document}


\title[Khovanov homology of positive links and of L-space knots]{Khovanov homology of positive links and of L-space knots}

\author[M.\ Kegel]{Marc Kegel}
\address{Universidad de Sevilla, Avda.\ Reina Mercedes s/n,
41012 Sevilla,
Spain}
\email{kegelmarc87@gmail.com}

\author[N.\ Manikandan]{Naageswaran Manikandan}
\address{Max Planck Institute for Mathematics in Bonn, Vivatsgasse 7,
53111 Bonn, Germany}
\email{manikandan@mpim-bonn.mpg.de}

\author[L.\ Mousseau]{Leo Mousseau}
\address{Universite de Neuchatel, Av. du Premier-Mars 26, 2000 Neuchatel, Switzerland}
\email{leo.mousseau@unine.ch}

\author[M.\ Silvero]{Marithania Silvero}
\address{Universidad de Sevilla, Avda.\ Reina Mercedes s/n,
41012 Sevilla,
Spain}
\email{marithania@us.es}


\date{\today} 


\begin{abstract}
In this article, we develop new obstructions for positivity of links coming from Khovanov homology.
More concretely, we show that the first Khovanov homology of a non-split, positive link is supported in a single quantum grading determined by the Seifert genus of the link, where the group is free abelian and of rank determined by the Seifert graph of any of its positive link diagrams. In particular, for a non-split, positive link, the first Khovanov homology vanishes if and only if the link is fibered. Moreover, we deduce that a non-split, positive link is fibered if and only if its second Khovanov homology contains no torsion elements.

These obstructions from Khovanov homology turn out to be strong; for example, they detect positivity for all prime knots with at most $12$ crossings. Moreover, we extend these results to $(p,q)$-cables of positive knots whenever $q\geq p$. We also show that several infinite families of Heegaard Floer L-space knots have vanishing first Khovanov homology. This suggests a possible extension of our results to L-space knots.
\end{abstract}

\makeatletter
\@namedef{subjclassname@2020}{%
  \textup{2020} Mathematics Subject Classification}
\makeatother


\maketitle


\section{Introduction}
A link is called positive if it admits a diagram in which every crossing is positive.
Although it is algorithmically decidable if a link is positive, this algorithm is not working well in practice, since it involves listing all positive diagrams up to a certain complexity (see for example Theorem~6.1 in~\cite{Stoimenow_pos_algo}). In this article, we develop new obstructions for positivity coming from Khovanov homology.

\begin{theorem} \label{thm:khovanov_positivity}
If $L$ is a non-split, positive link, then its integral Khovanov homology satisfies
\begin{align*}
Kh^{1,j}(L)=\begin{cases}
    \Z^{p_1(L)} &\textrm{ if } j=2-\chi(L),\\
    0 &\textrm{ otherwise,}
\end{cases}
\end{align*}

\begin{align*}
T\big(Kh^{2,j  }(L)\big)=\begin{cases}
     \Z_2^{p_1(L)} &\textrm{ if } j=4-\chi(L),\\
     0 &\textrm{ otherwise,}
\end{cases}
\end{align*}
where $\chi(L)$ denotes the maximal Euler characteristic\footnote{If $L$ is a knot, then $\chi(L)=1-2g(L)$, where $g(L)$ denotes the Seifert genus of $L$.} of a (possibly disconnected) Seifert surface for $L$ without closed components, $p_1(L)$ denotes the cyclomatic number of the reduced Seifert graph of any positive diagram of $L$, and $T(A)$ denotes the torsion subgroup of an abelian group~$A$.
\end{theorem}

For a split, positive link, the first Khovanov homology need not be supported in one quantum grading. Applying the Künneth theorem to the non-split components gives the precise replacement in Theorem~\ref{thm:split_positive}. For the definition of $p_1(L)$, see Section~\ref{sec:prelim}.

In general, Khovanov homology does not detect fiberedness. For example, $K9a12$ is fibered and $K11n83$ is not fibered, but both knots share the same Khovanov homology (after mirroring one of them)~\cite{KnotInfo}. Combining Theorem~\ref{thm:khovanov_positivity} with the fiberedness criterion for positive links in Theorem~\ref{thm:burdened_fibered+positive}, we prove that the first and second Khovanov homologies independently detect fiberedness among positive links.

\begin{theorem}\label{thm:pos_fib}
Let $L$ be a positive link. Then the following are equivalent.
\begin{enumerate}
    \item $L$ is fibered.
    \item $L$ is non-split and its first Khovanov homology vanishes, i.e.\ $Kh^{1,*}(L)=0$.
    \item $L$ is non-split and the torsion subgroup of the second Khovanov homology of $L$ vanishes, i.e.\ $T\big(Kh^{2,*}(L)\big)=0$.
\end{enumerate}
\end{theorem} 

Since it is algorithmically decidable if a link is fibered (and those algorithms are working well in practice)~\cite{Ni}, cf.~\cite{Gabai_fibered}, Theorems~\ref{thm:khovanov_positivity} and~\ref{thm:pos_fib} can also be regarded as an effective obstruction for positivity of links. In Section~\ref{sec:exp_data}, we discuss the strength of these results for obstructing positivity. 

Conceptually similar results to Theorem~\ref{thm:pos_fib} were obtained in~\cite{Stoi_MR2159224,Buch_arxiv.2204.03846,buchanan2023pair} for the Jones polynomial. However, our results concern the homological grading in Khovanov homology, whereas the Jones polynomial results can be interpreted as restrictions in the quantum grading, so the two types of obstruction are complementary. On the other hand, Sto\v{s}i\'c~\cite{Stosic} proved that braid positive links have trivial first Khovanov homology. Theorem~\ref{thm:pos_fib} generalizes Sto\v{s}i\'c's result since non-split braid positive links are necessarily fibered~\cite{Gabai_Murasugi} and positive. Nevertheless, there exist fibered positive links that are not braid positive. Among the prime knots with at most $12$ crossings, there are exactly $16$ such knots, the simplest one being the knot $K10n7$~\cite{KnotInfo}. We present in Proposition~\ref{prop:infinite_family} an infinite family of fibered, positive knots that are not braid positive.

\begin{corollary}\label{cor:concordance}
    If a knot $K_0$ is ribbon concordant to a positive knot $K_1$, then $Kh^{1,j}(K_0)$ is trivial if $j\neq2g(K_1)+1$ and is free abelian of rank at most $p_1(K_1)$ if $j=2g(K_1)+1$. Moreover, $T(Kh^{2,j}(K_0))$ is trivial if $j\neq2g(K_1)+3$ and is $\Z_2^k$ for some $k\leq p_1(K_1)$ if $j=2g(K_1)+3$. If $K_1$ is in addition fibered, then $Kh^{1,*}(K_0)=0$ and $Kh^{2,*}(K_0)$ is torsion-free.
\end{corollary}

\begin{proof}
    Levine--Zemke~\cite{LevineZemke} proved that a ribbon concordance induces a grading-preserving injection on integral Khovanov homology. This injection in homological degrees $1$ and $2$, together with Theorem~\ref{thm:khovanov_positivity} applied to $K_1$, yields the claimed statements.
\end{proof}

\subsection{Previously known positivity obstructions from Khovanov homology}\label{sec:other_obstructions}
Other obstructions of the Khovanov homology groups of positive links were known before. If $L$ is a non-split, positive link, then
\begin{align*}
    Kh^{i,j}(L)&=0 \,\textrm{ if } \,i<0 \,\textrm{ or }\, j<-\chi(L),\\
    Kh^{0,j}(L)&=\begin{cases}
        \Z &\textrm{ if } \,j=-\chi(L) \,\textrm{ and }\,2-\chi(L),\\
    0  &\textrm{ otherwise, } \\
    \end{cases}\\
    Kh^{i,-\chi(L)}(L)&=0 \, \textrm{ if }\,  i\neq0,\\
    Kh^{i,2-\chi(L)}(L)&=0 \, \textrm{ if } \, i\notin\{0,1\}. 
\end{align*}

The vanishing in negative homological grading and the formula for $Kh^{0,*}$ follow directly from the positive Khovanov complex, as observed and recorded in~\cite[Proposition~6.1]{Patterns_in_Khovanov_homology}. 
The last two formulas can be calculated from~\cite[Remark 6.4]{almost_extremeKH}, as positive links are $0$-adequate.
We have stated these formulas only for non-split links; the split case is recorded separately in Theorem~\ref{thm:split_positive}.
We refer to Figure~\ref{fig:positivity_table} for a schematic picture of the Khovanov homology of a non-split, positive link.
Some of these results can also be recovered from~\cite{semi_adequate_torsion,Extremal_KH_girth}.
In particular, we observe that Khovanov homology detects the genus among positive links. On the other hand, for every $k>0$ there exists a fibered, positive link $L$ such that $Kh^{i}(L)$ is non-trivial for every $i=0,2,3,4,\ldots,k$. As concrete examples we can take $2$-stranded torus links $T(2,k)$~\cite[Proposition~26]{Khovanov_homology}.

Other important results regarding the Khovanov homology of positive links were obtained in~\cite{Lee,Rasmussen}.

\begin{figure}[htbp]
\centering
\includegraphics[width = 7.3cm]{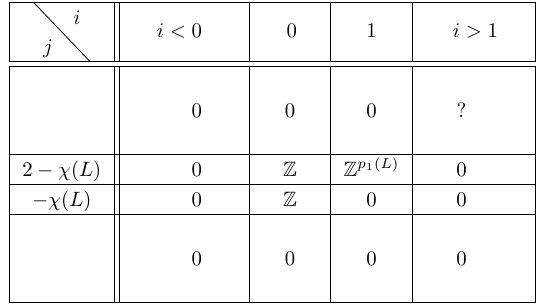}
\caption{Khovanov homology of a non-split, positive link $L$. }
\label{fig:positivity_table}
\end{figure}

\subsection{Khovanov homology of cables}
The skein long exact sequence for Khovanov homology, together with Theorems~\ref{thm:khovanov_positivity} and~\ref{thm:pos_fib} allow us to prove that the Khovanov homology of certain positive cables of positive knots looks like the Khovanov homology of a positive knot. Here, we emphasize that such a cable (even if its cabling slopes are positive) is, in general, not a positive link. We refer to Remark~\ref{rem:not_pos} for examples of this phenomenon.

\begin{theorem}\label{thm:KHcable}
    If $K$ is a positive knot, then for every $q\geq p\geq2$ the Khovanov homology of its $(p,q)$-cable $K_{p,q}$ satisfies
    \begin{align*}
    Kh^{i,j}(K_{p,q})&=0 \,\textrm{ if } \,i<0,\\
    Kh^{0,j}(K_{p,q})&=\begin{cases}
        \Z &\textrm{ if } \,j=1-\chi(K_{p,q})\pm1,\\
    0  &\textrm{ otherwise, } \\
    \end{cases}\\
    Kh^{1,j}(K_{p,q})&=\begin{cases}
        \Z^{p_1(K)} &\textrm{ if } \,j=2-\chi(K_{p,q}) ,\\
    0  &\textrm{ otherwise, } 
    \end{cases}\\
    T\big(Kh^{2,j}(K_{p,q})\big)&=\begin{cases}
        \Z_2^{p_1(K)} &\textrm{ if } \,j=4-\chi(K_{p,q}) ,\\
    0  &\textrm{ otherwise. } 
    \end{cases}
\end{align*}
\end{theorem}

In the special case when $K$ is fibered, the above theorem implies the following.

\begin{corollary}\label{cor:fibered_cable}
If $K$ is positive and $q\geq p\geq2$, then $K$ is fibered if and only if $Kh^{1,*}(K_{p,q})$ is trivial, if and only if $Kh^{2,*}(K_{p,q})$ is torsion free.
\end{corollary}

The proof uses a skein-exact-sequence induction in the spirit of Sto\v{s}i\'c's work; related skein inductions for torus links appear in~\cite{Ren_torus_links}.

\subsection{Khovanov homology of L-space knots}

Another interesting class of knots is given by L-space knots. Here, an \textit{L-space} knot is a knot that admits a positive surgery to a Heegaard Floer L-space~\cite{OzsvathSzabo}. L-space knots are necessarily fibered~\cite{Ni} and strongly quasipositive~\cite{Hedden}. Moreover, many (but not all) L-space knots are braid positive~\cite{Anderson_et_al, BakerKegel,Himeno} and thus are fibered and positive, so Theorem~\ref{thm:khovanov_positivity} applies. 

Certain cable L-space knots are known to not be positive (see Remark~\ref{rem:not_pos}), while it remains open whether a hyperbolic L-space knot can fail to be positive. Section~\ref{sec:comp} verifies our Khovanov positivity constraints for the Baker--Kegel family and for many cable L-space knots.

As a consequence of Theorem~\ref{thm:KHcable}, we get the following result.

\begin{corollary}\label{cor:KHL-spaceCable}
If $K$ is an L-space knot which is a cable of a positive knot, then the second Khovanov homology of $K$ contains no torsion elements and
   \begin{align*}
   Kh^{0,2g(K)\pm1}(K)=\Z\, \textrm{ and }\,
   Kh^{i,j}(K)= 0\,\textrm{ for all other groups with }\, i\leq1. 
\end{align*}
\end{corollary}

\begin{proof}    
By~\cite{Hom_cabling_and_Lspace_surgeries}, it is known that if $K$ is a cable L-space knot, then its cable parameters satisfy $q\geq p\geq2$. Since L-space knots are fibered~\cite{Ni}, Theorem~\ref{thm:KHcable} implies the result.
\end{proof}

 In light of these results, we conjecture that L-space knots have Khovanov homology that looks like the Khovanov homology of a positive knot.

\begin{conjecture}\label{conj:Lspace}
   If $K$ is an L-space knot, then the second Khovanov homology of $K$ contains no torsion elements and
   \begin{align*}
   Kh^{0,2g(K)\pm1}(K)=\Z\, \textrm{ and }\,
   Kh^{i,j}(K)= 0\,\textrm{ for all other groups with }\, i\leq1. 
\end{align*}
\end{conjecture}

We remark that the result of Levine--Zemke~\cite{LevineZemke} gives a potential way to study this conjecture. If we can show that any L-space knot $K$ is ribbon concordant to a fibered, positive (or a braid positive) knot, then $Kh^{0,2g(K)\pm1}(K)$ is $0$ or $\Z$ and all other Khovanov homology groups in homological grading at most $1$ vanish. On the other hand, we can search for a counterexample by finding a ribbon concordance from a knot $K_0$ with Khovanov homology that is not of the form of a positive knot to an L-space knot $K_1$. 

\subsection{Experimental data} \label{sec:exp_data}

Theorems~\ref{thm:khovanov_positivity} and~\ref{thm:pos_fib} can be seen as obstructions for a knot to be positive. From the behavior of Khovanov homology under mirroring a knot, we see that this also gives an obstruction for a knot being negative. This obstruction seems to be strong, since for low-crossing knots, Khovanov homology detects positivity. Among the $2977$ prime knots with at most $12$ crossings, exactly $246$ are positive or negative~\cite{KnotInfo}. Theorems~\ref{thm:khovanov_positivity} and~\ref{thm:pos_fib} together with the obstructions from Section~\ref{sec:other_obstructions} obstruct all the other $2731$ knots from being positive or negative~\cite{KnotInfo}. 

From this perspective, the more interesting class of knots is given by the census knots~\cite{Dunfield}, i.e.~the hyperbolic knots whose complements can be triangulated by at most $9$ ideal tetrahedra. Thus, these knots have simple complements. However, the simplest known diagrams of some census knots have almost $300$ crossings. In~\cite{CensusKnotInvariants}, the Khovanov homology for exactly $785$ of the $1267$ census knots was computed.  Among those $785$ knots, $319$ are positive or negative~\cite{CensusKnotInvariants}. To obstruct positivity (or negativity) for the remaining knots, we applied the obstruction from Khovanov homology. Again, this obstruction turns out to be strong, although not perfect. We can obstruct all but $17$ of the remaining knots from being positive or negative. We refer to Figure~\ref{fig:census_exampl} for an example. 
The data mentioned above can be accessed at~\cite{data}.

 \begin{figure}
     \centering
      \includegraphics[width=.69\textwidth]{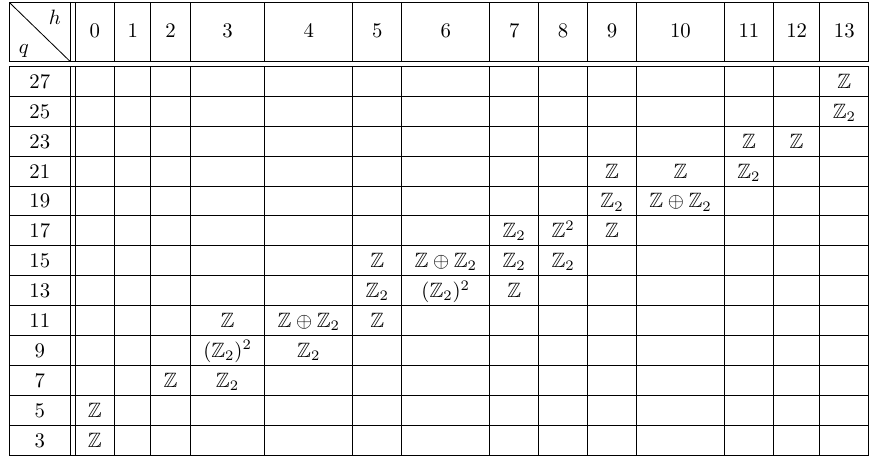} \includegraphics[width=.3\textwidth]{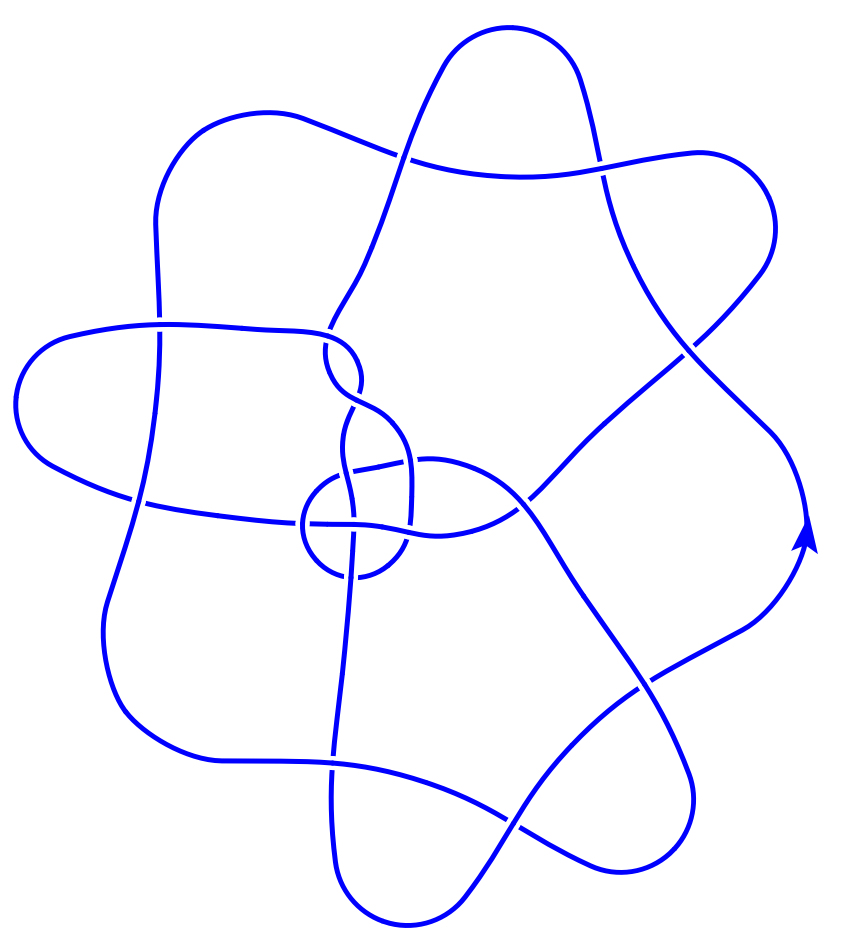} 
     \caption{A diagram of the mirror of the census knot $o9\_34097$ and its Khovanov homology table. Since it is not fibered, Theorem~\ref{thm:pos_fib} implies that it is not positive. }
     \label{fig:census_exampl}
 \end{figure}

 \subsection*{Conventions} 
 Since positivity of a link depends on the orientation of its components, for us, a link is the isotopy class of a smooth embedding of a disjoint union of finitely many oriented circles into $S^3$.
 
 We consider Khovanov homology groups up to isomorphism. Thus, whenever we write an expression like $Kh^{i,j}(L)=\Z$, we mean that $Kh^{i,j}(L)$ is isomorphic to $\Z$. $Kh^{i,j}(L)$ denotes Khovanov homology with integer coefficients. It will sometimes be useful to also work with more general coefficient groups. Whenever we write $R$ without further specifying it, $R$ denotes an arbitrary finitely presented abelian group. We denote by $Kh^{i,j}(L;R)$ the Khovanov homology groups with coefficients in $R$. 

\subsection*{Acknowledgment}
This project was initiated during a discussion between M.K.\ and M.S.\ at the semester program ``Braids'' (Feb 1- May 6, 2022) at the Institute for Computational and Experimental Research in Mathematics (ICERM). We thank ICERM for the invitation and financial support. The authors wish to thank Chris Wendl for his hospitality when M.S.\ visited the HU Berlin. We also thank Lukas Lewark for explaining to us why the $(2,1)$-cable of the trefoil is a strongly quasipositive, fibered knot with non-vanishing first Khovanov homology, which was the initiation for writing Section~\ref{sec:comp}. Special thanks go to the anonymous referee for an extremely detailed and useful report correcting several smaller errors and improving the exposition. In particular, the report has inspired the results about torsion in the second Khovanov homology of positive links.

\subsection*{Individual grant support}
M.K.\ is supported by a Ram\'on y Cajal grant (RYC20\-23-043251-I) and the project PID2024-157173NB-I00 funded by MCIN/AEI/10.13\-039/50100011033, by ESF+, and by FEDER, EU; and by a VII Plan Propio de Investigación y Transferencia (SOL2025-36103) of the University of Sevilla.
N.M.\ was funded by the Deutsche Forschungsgemeinschaft (DFG, German Research Foundation) under Germany's Excellence Strategy – The Berlin Mathematics Research Center MATH+ (EXC-2046/1, project ID: 390685689). N.M.\ acknowledges the Max-Planck Institute for Mathematics in Bonn for its hospitality and financial support.
M.S.\ is partially supported by the project PID2024-157173NB-I00 funded by MCIN/AEI/10.13039/501100011033 and by FEDER, EU.  


\section{Preliminaries}\label{sec:prelim}
Let $D$ be a link diagram and let $cr(D)=\{c_1, \ldots, c_n\}$ be the set of its crossings. A \textit{state} $s$ assigns a marker $0$ or $1$ to each crossing of $D$, that is $s \colon cr(D) \to \{0,1\}$. Let $\mathcal{S}(D)$ be the collection of $2^n$ possible states of $D$. We write $s_0$ for the state assigning a $0$-marker to every crossing. 

Given $s \in \cS(D)$, the \textit{resolution} $sD$ corresponds to the diagram obtained after smoothing every crossing $c_i \in cr(D)$ according to its marker $s(c_i)$ following Figure~\ref{fig:markers}(a). 

\begin{figure}[htbp]
\centering
\includegraphics[width = 11cm]{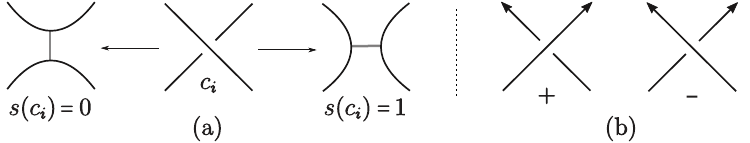}
\caption{The smoothing of a crossing according to its $0$ or $1$ marker and the sign (positive or negative) of a crossing are shown in (a) and (b), respectively.}
\label{fig:markers}
\end{figure}

The resolution $sD$ consists of a collection of $|sD|$ circles embedded in the plane together with some $0$- and $1$-{\it{chords}} (segments connecting two circles in the place where there was a crossing). See Figure~\ref{fig:resolution}(b). 

\begin{figure}[htbp]
\centering
\includegraphics[width = 11.9cm]{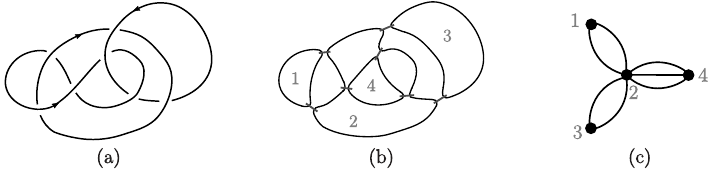}
\caption{A positive diagram $D$, its resolution $s_0D$, and the associated state graph $G_0(D)$ are shown in (a), (b), and (c), respectively.}
\label{fig:resolution}
\end{figure}

\begin{definition}\label{def:stategraph}
Given a state $s\in \cS(D)$, the \textit{state graph} $G_s(D)$ is a planar graph obtained by collapsing each circle of $sD$ to a vertex so that each chord in $sD$ becomes an edge in $G_s(D)$. The circles and chords of $sD$ are in bijection with the vertices and edges of $G_s(D)$. See Figure~\ref{fig:resolution}(c). 
\end{definition} 

The graph $G_s(D)$ may contain loops and multi-edges (i.e.~edges connecting the same pair of vertices, also called parallel edges). Given a graph $G$, its associated \textit{reduced graph} $G^{red}$ is the graph obtained from $G$ by identifying edges that connect the same pair of vertices. We write $G_0(D)$ for the state graph $G_{s_0}(D)$.

Recall that every crossing in an oriented diagram is either positive or negative, according to the convention shown in Figure~\ref{fig:markers}(b). An oriented link is \textit{positive} if it admits a positive diagram, i.e.~a diagram in which every crossing is positive. Observe that the $0$-smoothing of a positive crossing agrees with its \textit{Seifert smoothing} (that is, the only smoothing preserving orientation), and therefore the \textit{Seifert graph} of a positive diagram $D$ coincides with $G_0(D)$.

The cyclomatic number of a graph $G$ is the minimal number of edges that must be deleted to turn $G$ into a forest. It is $p(G)=e-v+k$, 
where $v$, $e$, and $k$ are the numbers of vertices, edges, and connected components. Given a positive link $L$, define $p_1(L)$ to be the cyclomatic number of the reduced Seifert graph of any positive diagram of $L$. Stoimenow's Jones polynomial formula~\cite[Theorem~3 and Corollary~2]{Stoi_MR2159224} shows that this is independent of the positive diagram. This argument also applies to split positive links. If \(L=L_1\sqcup\cdots\sqcup L_s\) is the split decomposition of a
positive link, then the reduced Seifert graph is the disjoint union of
the reduced Seifert graphs of the \(L_a\). Hence $p_1(L)=\sum_{a=1}^s p_1(L_a)$.

\begin{theorem}[Stoimenow~\cite{Stoi_MR2159224}; 
cf.~\cite{Cromwell,Futer,Buch_arxiv.2204.03846}]
\label{thm:burdened_fibered+positive}
Let $D$ be a positive diagram of a link $L$. Then $L$ is fibered if and only if
the reduced Seifert graph $G_0(D)^{\mathrm{red}}$ is a tree. \qed
\end{theorem}

As an example, we refer to Figure~\ref{fig:resolution} where the reduced Seifert graph $G_0(D)^{red}$ is a tree with 3 leaves, and thus $D$ represents a fibered link. 


\section{Khovanov homology}\label{sec:Khovdef}

In this section, we briefly review the definition of Khovanov homology and Viro's skein long exact sequences. We follow the approach in~\cite{Bar2002}, but first adapt the theory to have coefficients not in $\Q$ but in $\Z$. After that, we discuss how to take coefficients in a general finitely presented abelian group $R$.

Given a diagram $D$ of an oriented link $L$, we first define a chain complex $\lbracket D\rbracket$ which, after some grading shifts, becomes the Khovanov chain complex $\operatorname{CKh}(D)$.

Let $c_1, \ldots, c_n$ be the ordered crossings of $D$.
This induces an identification of the set $\mathcal{S}(D)$ of states with the elements in $\{0,1\}^n$ defined by $s = (s_1, \ldots, s_n) \in \{0,1\}^n$, with $s(c_j)=s_j$. Write $|s|=\Sigma_{j=1}^n s_j$. 
Consider the unit cube of dimension $n$, where each vertex consists of a state of $D$, and there is an oriented edge from the state $s$ to the state $s'$ if and only if they are identical except in their $k$-th coordinate, where $s_k=0$ and $s'_k=1$.
Let $V$ be a graded free abelian group of rank two with two generators, $v_+$ and $v_-$ of degrees $1$ and $-1$, respectively.
We associate to each vertex of the cube $s\in \{0,1\}^n$ the graded group \footnote{Given a graded group $W= \oplus_m W_m$, we denote by $\{l\}$ the \textit{degree shift} operation on $W$ by setting $W\{l\}_m=W_{m-l}$.} $V_s(D)=V^{\otimes|sD|}\{|s|\}$. Define a chain complex $\lbracket D\rbracket$ whose $r$-th chain group is defined as $\lbracket D\rbracket^r= \oplus_{|s|=r} V_s(D)$. 

Next, we define the differential $d^r\colon \lbracket D\rbracket^r \longrightarrow \lbracket D\rbracket^{r+1}$. Let $\xi$ be an edge of the cube oriented from a state $s$ to a state $s'$ satisfying $s(j)=s'(j)$ for all $j \neq k$, $s(k)=0$ and $s'(k)=1$, for $1\leq j,k\leq n$. Define $|\xi| $ of the edge $\xi$ as $|\xi|=|s|$ and $\operatorname{sgn}(\xi)= \Sigma_{j<k}s(j)$, with $k$ as before. 

The resolutions $sD$ and $s'D$ are identical except in a neighborhood of the $k$-th crossing, where either one circle in $sD$ \textit{splits} into two circles in $s'D$ or two circles in $sD$ \textit{merge} into one in $s'D$. The map $d_{\xi}\colon V_s(D) \longrightarrow V_{s'}(D)$ is defined as the identity on the tensor factors corresponding to the circles that are preserved and as either a multiplication $m\colon V \otimes V \rightarrow V$ or a comultiplication $\Delta\colon V \rightarrow V \otimes V$ on the factors corresponding to the circles involved in a merging or a splitting, respectively. These group homomorphisms are defined by:
$$
  \begin{array}{lllllllllll}
  m (v_+ \otimes v_+) = v_+ & & & & & &&& \Delta(v_+) = v_+ \otimes v_- + v_- \otimes v_+ \\ 
  m (v_+ \otimes v_-) = v_- & & & & & &&& \Delta(v_-) = v_- \otimes v_- \\
  m (v_- \otimes v_+) = v_- \\
  m (v_- \otimes v_-) = 0 \\
  \end{array}
$$

The differential of the complex is defined as $d^r = \Sigma_{|\xi|=r} (-1)^{\operatorname{sgn}(\xi)} d_{\xi}$. With the chain groups and differential defined above, $\lbracket D\rbracket$ is indeed a chain complex~\cite{Khovanov_homology}. We write $H^r(D)$ for its $r$-th homology group.  
The \textit{height shift} operation $[p]$ on a chain complex $\mathcal{A}$ is defined as $\mathcal{A}[p] = \mathcal{B}$, where $B^r=A^{r-p}$, with all differentials shifted accordingly. 

\begin{definition}\label{def:Khovanovcomplex}
Let $D$ be a diagram of an oriented link with $n_+$ positive and $n_-$ negative crossings, then the \textit{Khovanov chain complex} $\operatorname{CKh}(D)$ is defined as $\lbracket D\rbracket [-n_-] \{n_+-~2n_-\}$. We denote its $i$-th homology group by $Kh^i(D)$.
\end{definition}

Notice that each element in $\lbracket D\rbracket$ has two gradings: $r$ coming from the chain group and a second grading $q$ coming from the internal degree as an element in $V$. If we denote by $\lbracket D\rbracket^{r,q}$ the subset of elements of $\lbracket D\rbracket^r$ with second grading $q$, then $\lbracket D\rbracket^{r} =\bigoplus_{q \in \mathbb{Z} }\lbracket D\rbracket^{r,q}$
and $Kh^{i,j}(D) = H^{r, q}(D)$, with $r=i+n_-$ and $q=j-n_++2n_-$. We call such $i$ and $j$ {\it{homological}} and {\it{quantum}} gradings, respectively.  

\begin{theorem}[Khovanov~\cite{Khovanov_homology}]\label{thm:khovanovhom}
For every $i, j \in \Z$, the isomorphism types of the group $Kh^{i,j}(D)$ depend only on the isotopy class of the oriented link $L$. These groups are called the \textit{Khovanov homology groups} of $L$. \qed
\end{theorem}

Slightly more strongly, the bigraded chain-homotopy type of $\operatorname{CKh}(D)$ is already an invariant of the oriented link; see~\cite{Khovanov_homology,Bar2002}. Consequently, any construction depending only on that chain-homotopy type is a link invariant.

Let $R$ be a finitely presented abelian group. Since the chain groups of $\operatorname{CKh}(D)$ are free abelian, tensoring the complex and its differential with $R$ gives a bigraded complex. Its homology groups, denoted $Kh^{i,j}(D;R)$, are invariants of the oriented link represented by $D$.

In the proof of Theorem~\ref{thm:KHcable}, we will need Viro's long exact sequences for Khovanov homology \cite{Viro}, which we briefly recall here.

\begin{lemma}\label{lem:les}
Let $v$ be a negative crossing in an oriented link diagram $D$, and consider the diagram $D_0$ $($resp.~$D_1$$)$ obtained from $D$ by smoothing the crossing $v$ following a $0$-marker $($resp.~$1$-marker$)$. Assume $D_1$ is oriented in the only way preserving the orientation of $D$, and orient $D_0$ in such a way that the orientation of the components not involved in $v$ is preserved. Write $w$ and $w_0$ for the writhes of $D$ and $D_0$, respectively. Then there is a long exact sequence in Khovanov homology of the form
\begin{alignat*}{11}
\cdots & \rightarrow   && Kh^{i,j+1}(D_1;R)  &&  \rightarrow && Kh^{i,j}(D;R) \, && \rightarrow && \quad Kh^{\frac{w_0-w+1}{2}+i,\frac{3(w_0-w)+1}{2}+j}(D_0;R)\\
& \rightarrow && Kh^{i+1,j+1}(D_1;R) && \rightarrow \cdots && &&
\end{alignat*}
\end{lemma}

\begin{lemma} \label{lemma:les_positive}  
		Let $v$ be a positive crossing in an oriented link diagram $D$, and consider the diagram $D_0$ $($resp.~$D_1$$)$ obtained from $D$ by smoothing the crossing $v$ following a $0$-marker $($resp.~$1$-marker$)$. Assume $D_0$ is oriented in the only way preserving the orientation of $D$, and orient $D_1$ in such a way that the orientation of the components not involved in $v$ is preserved. Write $w$ and $w_1$ for the writhes of $D$ and $D_1$, respectively. Then there is a long exact sequence in Khovanov homology of the form
	\begin{alignat*}{11}
			\cdots & \rightarrow   && Kh^{\frac{w_1-w-1}{2}+i,\frac{3(w_1-w)-1}{2}+j}(D_1;R)  &&  \rightarrow && Kh^{i,j}(D;R) \quad && \rightarrow && Kh^{i,j-1}(D_0;R) \\
			& \rightarrow && Kh^{\frac{w_1-w-1}{2}+i+1,\frac{3(w_1-w)-1}{2}+j}(D_1;R) && \rightarrow \cdots &&
	\end{alignat*}  
\end{lemma}

Both follow directly from Viro's long exact sequences~\cite{Viro}. Since we could not find a proof for general coefficients in the literature, we include a brief proof sketch below.

\begin{proof}[Proof sketches of Lemmas~\ref{lem:les} and~\ref{lemma:les_positive}]
 Viro~\cite{Viro} identifies a split short exact sequence relating the chain complexes of $D$, $D_1$, and $D_0$. Because all chain groups are free abelian and the sequence is split, tensoring with an arbitrary abelian group $R$ yields again a short exact sequence of chain complexes. This induces a long exact sequence in homology, which (after the corresponding grading shifts) takes exactly the displayed forms.
\end{proof}

\begin{corollary}\label{cor:Bsmoothing}
    In the setting of Lemma~\ref{lem:les}, assume that $D_0$ is a diagram of the unknot, and set $u:=\frac{w-w_0-1}{2}$. Then the Khovanov homologies of the links presented by $D$ and $D_1$ are related as follows.
    \begin{itemize}
        \item[(1)] If $(i,j)\notin \big\{(u,3u+1\pm1),(u+1,3u+1\pm1)\big\}$ then
        \begin{equation*}
            Kh^{i,j}(D;R)=Kh^{i,j+1}(D_1;R).
        \end{equation*}
        \item[(2)] The other Khovanov homology groups fit into the following two exact sequences
        \begin{alignat*}{6}
            0&\rightarrow Kh^{u,3u+2\pm1}(D_1;R)&&\rightarrow Kh^{u,3u+1\pm1}(D;R)&&\rightarrow R\\
            &\rightarrow Kh^{u+1,3u+2\pm1}(D_1;R)&&\rightarrow Kh^{u+1,3u+1\pm1}(D;R)&&\rightarrow 0.
        \end{alignat*} 
    \end{itemize}
\end{corollary}

\begin{proof}
The statement follows from Lemma~\ref{lem:les} and the fact that the Khovanov homology of the unknot is $R$ in gradings $(i,j)=(0,\pm1)$ and zero for all other gradings. 
\end{proof}

Later, we will use the Künneth theorem and, in particular, the behavior under adding a split unknot.

\begin{lemma}\label{lem:kunneth}
	Let \( L \sqcup \bigcirc \) be the split union of an oriented link \( L \) with an unknot. Then for every \( i \in \mathbb{Z} \),  
	\[
		Kh^{i}(L \sqcup \bigcirc;R) = Kh^{i}(L;R) \oplus Kh^{i}(L;R).
	\]  
	In particular, if \( Kh^{i}(L;R) \) vanishes, then \( Kh^{i}(L \sqcup \bigcirc;R) \) vanishes as well. 
\end{lemma}  

\begin{proof}
The result follows from the Künneth formula in Khovanov homology~\cite[Corollary~12]{Khovanov_homology} together with the fact that $Kh^{i,j}(\bigcirc;R)$ equals $R$ if $(i,j)=(0,\pm1)$ and is zero otherwise.
\end{proof}


\section{Khovanov homology of positive links}\label{sec:main_proof}

This section is devoted to proving Theorems~\ref{thm:khovanov_positivity} and~\ref{thm:pos_fib}. We also show that the latter is a proper generalization of the result in~\cite[Theorem 3.1]{Stosic} by presenting an infinite family of fibered, positive knots which are not braid positive.

\subsection{First Khovanov homology of positive links}\label{sec:firstKH}

We start by proving the results for the Khovanov homology in homological degree one. We will deduce the corresponding statements in Theorems~\ref{thm:khovanov_positivity} and~\ref{thm:pos_fib} from the following more technical statement. For getting the torsion statements about the second Khovanov homology, it will be useful to work with more general coefficient groups.

 \begin{theorem}\label{thm:technical}
Let $D$ be a positive diagram of a link $L$, and let $D'$ be a positive diagram of a link $L'$ such that $G_0(D')=G_0(D)^{red}$. Then, 
\[
Kh^1(L;R)= Kh^1(L';R).
\]
Here, the isomorphism preserves the homological grading but need not preserve the quantum grading.
 \end{theorem}

We show in Lemma~\ref{lem:existence} below that such a diagram $D'$ always exists. However, neither $D'$ nor $L'$ is unique, as the Seifert graph $G_0(D')$ does not contain information on how the Seifert circles are embedded in the plane or how the various chords are attached to them. But Theorem~\ref{thm:technical} tells us that the first Khovanov homology does not depend on that extra information; however, the exact quantum grading does, cf.~Theorem~\ref{thm:khovanov_positivity}. For a concrete example, we refer to Figure~\ref{fig:example_proof}.

\begin{proof}[Proof of Theorem~\ref{thm:technical}]
First, we prove the theorem for integral coefficients.
We recall from Section~\ref{sec:Khovdef} that, for any diagram $D$, the groups $Kh^i(D)$ and $H^j(D)$ are related by (homological) grading shifts determined only by the number of negative crossings of $D$. 
Since the diagrams $D$ and $D'$ are positive, we have that $Kh^1(D) = H^1(D)$ and $H^1(D') = Kh^1(D')$.

Thus, we will analyze the first homology group $H^1(D)$ of the complex $\lbracket D \rbracket$. For that, it will be enough to consider $d^1\colon\lbracket D \rbracket^1\rightarrow \lbracket D \rbracket^2$. First, we describe these chain groups and $d^1$ in more detail. Since our statement is independent of the quantum grading, we can (for ease of notation) omit the degree shift operation on the chain groups.

Consider $G = G_0(D)$ and write $E=\{e_1, \ldots, e_n\}$ for its edges, where $e_j$ is associated with the crossing $c_j$, and $l=|s_0D|$ for its number of vertices. The states contributing to $\lbracket D \rbracket^1$ are those assigning a $1$-marker to a single crossing, and therefore 
$$\lbracket D \rbracket^1 = \bigoplus_{i=1}^{n} A_{i}$$
where the summand $A_{i}$ is associated to the state $s^i$ assigning a $0$-marker to every crossing but $c_i$. Since $D$ is positive and therefore $0$-adequate, we get that $A_{i}$ is isomorphic to $ V^{\otimes(l-1)}$, for every $1 \leq i \leq n$. 

Generators in $\lbracket D \rbracket^2$ are those states assigning exactly two $1$-markers to the crossings of $D$, and therefore
$$\lbracket D \rbracket^2 = \bigoplus_{\substack{i,j=1 \\ i<j}}^n A_{ij}$$
where the summand $A_{ij}$ corresponds to the state $s^{ij}$ assigning a $1$-marker exactly to the crossings $c_i$ and $c_j$. If $e_i$ and $e_j$ are parallel (i.e.\ they connect the same pair of vertices) in $G$, then $A_{ij}$ is isomorphic to $V^{\otimes l}$, for every $1 \leq i < j \leq n$.

The differential $d^1 \colon \lbracket D \rbracket^1 \to \lbracket D \rbracket^2$ is given as a direct sum of maps $\delta_{kij} \colon A_{k} \to A_{ij}$ corresponding to the edges of the cube of resolutions. 
Whenever $k\neq i,j$, it turns out that $\delta_{kij}$ is the zero map (equivalently, there is no edge connecting the states $s^k$ and $s^{ij}$ in the cube). 
Given two parallel edges $e_i$ and $e_j$ in $G$, with $i<j$, consider the following diagram.

\hfill\\
\begin{xy}
 (0,90)*+{}="empt";(30,90)*+{\lbracket D \rbracket^1}="D1"; (90,90)*+{\lbracket D \rbracket^2}="D2";%
 (30,70)*+{A_i\oplus A_j}="AiAj"; (90,70)*+{A_{ij}}="Aij";%
 (30,50)*+{V^{\otimes(l-1)}\oplus V^{\otimes(l-1)}}="VxV"; (90,50)*+{V^{\otimes l}}="V";%
 (60,80)*+{\circlearrowright}="circ1";%
 (60,60)*+{\circlearrowright}="circ2";%
 {\ar@{->}^{d^1} "D1";"D2"};
 {\ar@{->}_{\pi_i\oplus \pi_j} "D1";"AiAj"};
 {\ar@{->}^{\pi_{ij}} "D2";"Aij"};
 {\ar@{->}^{\delta_{iij}\oplus\delta_{jij}} "AiAj";"Aij"};
 {\ar@{->}_{\phi_i\oplus\phi_j}^{\cong} "AiAj";"VxV"};
 {\ar@{->}_{\cong}^{\phi_{ij}} "Aij";"V"};
 {\ar@{->}_{\overline{\delta}_{iij}\oplus\overline{\delta}_{jij}} "VxV";"V"};
\end{xy}\\ 

Here $\pi_i$, $\pi_j$, and $\pi_{ij}$ denote the projections to the summands $A_i$, $A_j$, and $A_{ij}$. Observe that the smoothed diagrams $s^iD$, $s^jD$ and $s ^{ij}D$ are identical in $(l-2)$ Seifert circles. 
Order the circles in these smoothed diagrams such that the $p$-th circle coincides in the three of them, for $1\leq p \leq l-2$. These orderings induce preferred isomorphisms denoted by $\phi_i$, $\phi_j$, and $\phi_{ij}$. For $k=i,j$ we define $$\overline \delta_{kij}:=\phi_{ij}\circ\delta_{kij}\circ\phi_k^{-1}.$$

By construction, the diagram above commutes. 
The map $\delta_{iij}$ (resp.~$\delta_{jij}$) is given by the identity map on all factors except the $(l-1)$-th one, corresponding to the last circle in the chosen order, where it acts as the comultiplication map $-\Delta$ (resp.~$\Delta$). Recall that the sign is determined by the condition $i<j$. Thus we conclude 
\begin{alignat*}{11}
    &\overline{\delta}_{iij}\colon v_1\otimes\cdots\otimes v_{l-1} &&\longmapsto &&-&&v_1\otimes\cdots\otimes v_{l-2}\otimes \Delta(v_{l-1}),\\
    &\overline{\delta}_{jij}\colon v_1\otimes\cdots\otimes v_{l-1} &&\longmapsto&&&& v_1\otimes\cdots\otimes v_{l-2}\otimes \Delta(v_{l-1}).
\end{alignat*}
 Now let $x$ be an element in the kernel of $d^1$. Since for $k\neq i,j$ the maps $\delta_{kij}$ are zero and $\Delta$ is injective, it follows from the commutativity of the diagram that 
$$\phi_i\circ \pi_i(x)=\phi_j\circ \pi_j(x).$$
Thus we have shown that the kernel of $d^1$ is contained in the subgroup given by elements $(x_1,\ldots,x_n)\in \lbracket D \rbracket^1$ such that $\phi_r(x_r)=\phi_s(x_s)$ if $e_r$ is parallel to $e_s$. Since the $\phi_*$ are isomorphisms, identifying coordinates associated with parallel edges preserves the first homology groups. Thus we get that $H^1(D)=H^1(D')$, where $D'$ is a diagram satisfying $G_0(D')=G^{red}$. By the discussion from the beginning of this proof, this implies $Kh^1(L)= Kh^1(L')$, which establishes the theorem for integral coefficients.

For coefficients in a finitely presented abelian group $R$, we tensor
the entire commutative diagram above with $R$. Since
all chain groups are free abelian, the resulting tensored chain groups are
direct sums of copies of $R$, and the maps are represented, with respect to the induced bases, by the
same matrices as before.

In particular, with respect to the standard basis $(v_+,v_-)$ of $V$ and the standard basis
$(v_+\otimes v_+,\,v_+\otimes v_-,\,
  v_-\otimes v_+,\,v_-\otimes v_-)$ of $V\otimes V$, the comultiplication is represented by
\[
\Delta=
\begin{pmatrix}
0&0\\
1&0\\
1&0\\
0&1
\end{pmatrix}.
\]
Consequently, after tensoring with $R$, it is given by
\[
\Delta_R\colon R^2\longrightarrow R^4,
\qquad
(a,b)\longmapsto(0,a,a,b),
\]
and is therefore injective. Hence, the same argument as for integral
coefficients shows that, if
$x\in\ker d^1_R$, then
\[
\phi_i\circ\pi_i(x)=\phi_j\circ\pi_j(x)
\]
whenever $e_i$ and $e_j$ are parallel and thus the same proof as above implies that $Kh^1(L;R)= Kh^1(L';R)$.
\end{proof}

In the following two lemmata, we show that $D'$ from Theorem~\ref{thm:technical} can be chosen so that it satisfies some extra conditions. 

\begin{lemma}\label{lem:existence}
The positive diagram $D'$ of $L'$ in Theorem~\ref{thm:technical} can be chosen to be alternating and such that its signature\footnote{Here we use the convention that the signature of a positive link is positive.} is $\sigma(L')=p_1(L)$.
\end{lemma}

\begin{proof} 
Since $G_0(D)$ is the Seifert graph of $D$, it is planar and bipartite (and so is $G_0(D)^{red}$). We first construct a positive diagram $D'$ such that $G_0(D')= G_0(D)^{red}$. To do so, just choose a planar embedding of $G_0(D)^{red}$ and replace every vertex with a circle and every edge connecting two vertices by a positive crossing connecting the two corresponding circles. By construction, the obtained diagram $D'$ is positive and its Seifert graph equals $G_0(D)^{red}$. Since $G_0(D')$ is bipartite, a simple combinatorial argument shows that $D'$ is alternating. We refer to Figure~\ref{fig:example_proof} for an example.

In order to compute the signature of $L'$, we reduce the diagram $D'$ and, by abuse of notation, call this reduced alternating diagram again $D'$. We use the result in~\cite{Traczyk} stating that any link $L_0$ obtained from a link $L$ by performing a Seifert smoothing of a positive crossing in a reduced alternating diagram of $L$ satisfies $\sigma(L)=\sigma(L_0)+1$. Since $D'$ is alternating, performing a Seifert smoothing in any subset of its crossings produces an alternating diagram. Moreover, the effect of such a smoothing of a crossing corresponds to deleting the associated edge in the Seifert graph. Hence, there exists a sequence of $p_1(L)$ Seifert smoothings transforming $D'$ into a diagram whose Seifert graph is a forest and thus represents an unlink. Since the signature of an unlink vanishes, we conclude that $\sigma(L')=p_1(L)$.
\end{proof}

\begin{figure}[htbp]
\centering
\includegraphics[width = 11.9cm]{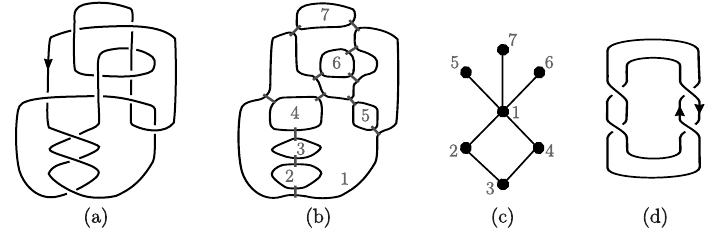}
\caption{A positive diagram of knot $K12n110$, its Seifert smoothing, and its reduced Seifert graph with $p_1=1$ are shown in (a), (b), and (c), respectively. The diagram $D'$ from Lemma~\ref{lem:existence} has the same reduced Seifert graph. In (d), we show $D'$ after removing nugatory crossings, from which we see that it represents the mirror of $L4a1\{0\}$. We compute that $Kh^1(K12n110)=\Z$ is supported in quantum grading $7$ and $Kh^1(m(L4a1\{0\}))=\Z$ is supported in quantum grading $2$.}
\label{fig:example_proof}
\end{figure}

\begin{lemma}\label{lem:last_group}
If $L$ is a non-split, positive link and if $D'$ is a diagram of the link $L'$ as constructed in Lemma~\ref{lem:existence}. Then
$$Kh^{1,p_1(L)+3}(L')=0.$$
\end{lemma}

\begin{proof}
If $D'$ is not reduced, we observe that its reduction still has the same $p_1$. Thus, without loss of generality, we can assume $D'$ to be reduced. We will prove the result by induction on $p_1(L)=p_1(D')$. If $p_1(D')=0$, the diagram $D'$ represents an unlink, which has vanishing first Khovanov homology. If $p_1(D')>0$, we consider one of its crossings $v$, whose resolution will decrease $p_1$ of the graph, and consider its $0$-smoothing $D_0$ and its $1$-smoothing $D_1$. The long exact sequence from Lemma~\ref{lemma:les_positive} yields an exact sequence of the form
    \begin{align*}
        Kh^{\frac{w_1-w'-1}{2}+1,\frac{3(w_1-w')-1}{2}+p_1(D')+3}(D_1)\rightarrow Kh^{1,p_1(D')+3}(D')\rightarrow Kh^{1,p_1(D_0)+3}(D_0),
    \end{align*}
    where $w'$ and $w_1$ denote the writhes of $D'$ and $D_1$ respectively. 
    
    To show the vanishing of the middle group, it suffices to demonstrate that the groups on the left and right vanish. For that, we observe that $D_0$ is again a positive and alternating diagram with $p_1(D_0)=p_1(D')-1$ and such that its reduced Seifert graph agrees with its Seifert graph. Thus by induction hypothesis $Kh^{1,p_1(D_0)+3}(D_0)=0$. 

    $D_1$ is in general not positive, but it is always alternating and, in particular, it is adequate. The graph $G_0(D_1)$ is obtained from the Seifert graph of $D'$ by deleting one edge and identifying the corresponding vertices. Thus $G_0(D_1)$ is connected and not bipartite. Moreover, since $D_1$ is alternating and reduced, it minimizes the crossing number and therefore, represents a non-split link. Then we can apply Remark~6.4 from~\cite{almost_extremeKH} to deduce that 
    $$Kh^{\frac{w_1-w'-1}{2}+1,\frac{3(w_1-w')-1}{2}+p_1(D')+3}(D_1)=0,$$
    which proves the lemma.
\end{proof}

We are now ready to prove the statements about the first Khovanov homologies in Theorems~\ref{thm:khovanov_positivity} and~\ref{thm:pos_fib}. The statements about torsion of the second Khovanov homology we will prove in the next subsections.

\begin{proof}[Proof of Theorem~\ref{thm:khovanov_positivity} (first Khovanov homology)]
We conclude from Lemma~\ref{lem:existence} and Theorem~\ref{thm:technical} that there exists a reduced alternating and positive diagram $D'$ of a link $L'$ with signature $\sigma(L')=p_1(L)$ such that $Kh^1(L)=Kh^1(L')$. Because $L$ is non-split, its Seifert graph, and hence $G_0(D')$, is connected. After removing nugatory crossings, $D'$ is a connected reduced alternating diagram and therefore represents a non-split link. Since $L'$ is non-split and alternating, its Khovanov homology is thin and based around the signature~\cite{Lee}, i.e.\ 
    $$Kh^{1,j}(L')=0\,\,\textrm{ if }\,\,j\neq p_1(L)+1, p_1(L)+3.$$
From Lemma~\ref{lem:last_group} we also know that $Kh^{1,p_1(L)+3}(L')$ is trivial and thus the only potential non-trivial Khovanov homology group of $L'$ in homological degree $1$ has quantum grading $p_1(L)+1$. To compute this group, we observe that the minimal quantum grading from~\cite{almost_extremeKH} is $e'-v' = p_1(L)-1,$ where $e'$ and $v'$ denote the number of edges and vertices in the Seifert graph $G_0(D')$. Since any Seifert graph is bipartite, it follows from~\cite[Remark 6.4]{almost_extremeKH} that $Kh^{1,p_1(L)+1}(L')$ is isomorphic to $\Z^{p_1(L)}$. In conclusion, we have computed that
\begin{align*}
     Kh^{1,j}(L')=\begin{cases}
      \Z^{p_1(L)} & \text{ if } j= {p_1(L)+1},\\
     0 & \text{otherwise}.
     \end{cases}
 \end{align*}
From Theorem~\ref{thm:technical} we conclude directly that $Kh^1(L)=\Z^{p_1(L)}$. To get the statement in Theorem~\ref{thm:khovanov_positivity} about the quantum grading, we apply again~\cite[Remark 6.4]{almost_extremeKH} to the positive diagram $D$ to deduce that $Kh^{1,2-\chi(L)}(L)=\Z^{p_1(L)}$ and thus all other Khovanov homology groups in homological degree $1$ are trivial.
\end{proof}

\begin{proof}[Proof of Theorem~\ref{thm:pos_fib} (first Khovanov homology)]
If $L$ is fibered, then $L$ is non-split, see for example~\cite{Gabai_fibered}. Thus, Theorem~\ref{thm:burdened_fibered+positive} implies that diagram $D'$ from Theorem \ref{thm:technical} can be chosen to be a diagram of the unknot, and therefore $Kh^1(L)=Kh^1(D')$ is trivial. 
   
Conversely, suppose that $L$ is non-split and $Kh^1(L)$ is trivial. Theorem~\ref{thm:khovanov_positivity} gives $p_1(L)=0$, so the reduced Seifert graph is a tree. Theorem~\ref{thm:burdened_fibered+positive} then implies that $L$ is fibered.
\end{proof}

\subsection{Homological algebra} 

In this section, we record some basic results from homological algebra that we will need to prove the results about torsion in the second Khovanov homology of a positive link. 

\begin{lemma}\label{lem:UCT} Let $L$ be an oriented link.
\begin{enumerate}
    \item If $Kh^{1,j}(L)=\Z^{n}$ for some $n\in\N_0$, then $Kh^{1,j}(L;R)=R^{n}$ for every finitely presented abelian group $R$, if and only if $T(Kh^{2,j}(L))=0$. 
    \item If $Kh^{1,j}(L)=0$, then $Kh^{1,j}(L;R)=T_2(R)^n$ for every finitely presented abelian group $R$, if and only if $T(Kh^{2,j}(L))=\Z_2^n$.
\end{enumerate}
\end{lemma}

\begin{proof}
    First, we recall that since the boundary map in the Khovanov chain complex increases the $i$ grading, the universal coefficient theorem in the setting of Khovanov homology takes the form
    \begin{equation*}
        Kh^{i,j}(L;R)=\operatorname{Tor}\big( Kh^{i+1,j}(L),R\big) \oplus \big(Kh^{i,j}(L) \otimes R\big).
    \end{equation*}
    If $Kh^{1,j}(L)=\Z^{n}$ this yields
    \begin{equation*}
        Kh^{1,j}(L;R)=\operatorname{Tor}\big( Kh^{2,j}(L),R\big) \oplus R^n.
    \end{equation*}
    Since a finitely presented abelian group $A$ contains no torsion element if and only if $\operatorname{Tor}(A,R)$ is trivial for every finitely presented abelian group $R$, we deduce $(1)$ from the classification of finitely presented abelian groups.

    In the case $Kh^{1,j}(L)=0$ this simplifies to $Kh^{1,j}(L;R)=\operatorname{Tor}\big( Kh^{2,j}(L),R\big)$. If $T(Kh^{2,j}(L))=\Z_2^n$, we get
    \begin{equation*}
        Kh^{1,j}(L;R)=\operatorname{Tor}\big( Kh^{2,j}(L),R\big) =\operatorname{Tor}\big( \Z_2^n,R\big)=T_2(R)^n.
    \end{equation*}

If conversely, $Kh^{1,j}(L;R)= T_2(R)^n$
for every finitely presented abelian group $R$, we can in particular set
$R=\Z_p$ for a prime $p>2$ and get
\begin{equation*}
    0=T_2(\Z_p)^n
    = Kh^{1,j}(L;\Z_p)
    =
    \operatorname{Tor}\bigl(Kh^{2,j}(L),\Z_p\bigr).
\end{equation*}
Thus, the classification of finitely presented abelian groups implies
that there exist $k_1,\ldots,k_l\in\N$ such that
\begin{equation*}
    T\bigl(Kh^{2,j}(L)\bigr)
    =
    \Z_{2^{k_1}}\oplus\cdots\oplus\Z_{2^{k_l}}.
\end{equation*}
Next, we set $R=\Z_2$, which implies
\begin{align*}
    \Z_2^n
    = T_2(\Z_2)^n
    =
    \operatorname{Tor}\bigl(Kh^{2,j}(L),\Z_2\bigr)
    =
    \bigoplus_{i=1}^l
    \operatorname{Tor}\bigl(\Z_{2^{k_i}},\Z_2\bigr)
    = \Z_2^l.
\end{align*}
Hence $l=n$. Finally, setting $R=\Z_4$ gives
\begin{align*}
    \Z_2^n
    = T_2(\Z_4)^n
    =
    \bigoplus_{i=1}^n
    \Z_{\gcd(2^{k_i},4)}.
\end{align*}
It follows that $\gcd(2^{k_i},4)=2$ for every $i$, and hence
$k_i=1$ for every $i$. Therefore
$T\bigl(Kh^{2,j}(L)\bigr)=\Z_2^n$, 
as required.
\end{proof}

\subsection{Torsion in second Khovanov homology and general coefficients}

In this section, we will finish the proof of Theorems~\ref{thm:khovanov_positivity} and~\ref{thm:pos_fib}, i.e.\ we prove the results about the torsion in the second Khovanov homology. Along the way, we will also derive for each finitely presented abelian group $R$ an obstruction to positivity from its first Khovanov homology and argue that all these potentially infinitely many obstructions are combined in the obstruction coming from torsion in the second integral Khovanov homology. We start by identifying torsion in the second Khovanov homology in a certain quantum grading.

\begin{lemma}\label{lem:PS_torsion}
Let $L$ be a non-split, positive link. Then
\[
T\bigl(Kh^{2,4-\chi(L)}(L)\bigr)=\Z_2^{p_1(L)}.
\]
\end{lemma}

\begin{proof}
This is essentially contained in~\cite[Proposition~4.8(1)]{semi_adequate_torsion}. Since the notation there differs from our notation, we provide a more detailed argument.

Let $D$ be a positive diagram with $n$ crossings and $s=|s_0D|$. We apply~\cite[Proposition~4.8(1)]{semi_adequate_torsion} to the mirror diagram, so that the all-positive state graph in the convention of that paper is $G_0(D)$. Its reduced graph is connected and bipartite, and the cited proposition gives $\Z_2^{p_1(L)}$ in framed bidegree $(n-4,n+2s-8)$.
Their positive markers correspond to our $0$-markers, and their negative markers correspond to our $1$-markers. Thus, a framed bidegree $(a,b)$ corresponds in our unshifted complex to
$\left(\frac{n-a}{2},\frac{n-b}{2}\right)$.
The displayed group consequently lies in unshifted bidegree $(2,4-s)$. Since $D$ is positive, the oriented quantum shift is $n$, giving quantum grading $n-s+4=4-\chi(L)$. 
\end{proof}

Next, we present the obstruction for positive links in the first Khovanov homology with arbitrary coefficients.

\begin{theorem} \label{thm:khovanov_positivity_general}
	Let $L$ be a non-split, positive link. Then the Khovanov homology of $L$ with coefficients in $R$ satisfies
	\begin{align*}
		Kh^{1,j}(L;R)=\begin{cases}
			R^{p_1(L)} &\textrm{ if } j=2-\chi(L),\\
			{T_2(R)}^{p_1(L)}  &\textrm{ if } j=4-\chi(L), \\
			0 &\textrm{ otherwise,}
		\end{cases}
	\end{align*}
	where $T_2(R)$ denotes the $2$-torsion subgroup of $R$, i.e.\ $T_2(R) = \{x \in R \mid 2x = 0\}$.
\end{theorem}

\begin{ex}
    The referee has pointed out that the twist knot $K5a1$ is positive with $p_1(K5a1)=1$. Thus, we get    
    $Kh^1(K5a1;\Z_2)=\Z_2^2$, where the second $\Z_2$-summand is exactly the $T_2(\Z_2)$ summand in Theorem~\ref{thm:khovanov_positivity_general}.
\end{ex}

\begin{proof}[Proof of Theorem~\ref{thm:khovanov_positivity_general}]
The integral calculation already proved gives a copy of $R^{p_1(L)}$ in quantum grading $2-\chi(L)$ after applying Lemma~\ref{lem:UCT}. Lemma~\ref{lem:PS_torsion} and~\ref{lem:UCT} give a copy of
$T_2(R)^{p_1(L)}$
in quantum grading $4-\chi(L)$.

It remains to exclude additional summands. Choose the alternating positive diagram $D'$ from Lemma~\ref{lem:existence}. Theorem~\ref{thm:technical} applies over $R$. The integral Khovanov homology of the non-split alternating link $L'$ is thin~\cite{Lee}; moreover, torsion of a homologically thin link lies on the lower diagonal~\cite{Shumakovitch2012}. Lemma~\ref{lem:last_group} and~\cite[Proposition~4.8]{semi_adequate_torsion}, followed by Lemma~\ref{lem:UCT}, therefore give
\[
Kh^1(L';R)= R^{p_1(L)}\oplus T_2(R)^{p_1(L)}
\]
as an ungraded group. Thus, the two summands that we have already found with the correct quantum grading yield the full group.
\end{proof}

Now, we are ready to complete the proof of Theorems~\ref{thm:khovanov_positivity} and~\ref{thm:pos_fib}.

\begin{proof}[Proof of Theorem~\ref{thm:khovanov_positivity} (torsion in second Khovanov homology)]
Lemma~\ref{lem:PS_torsion} says that $Kh^{2,{4-\chi(L)}}(L)$ has torsion $\Z_2^{p_1(L)}$. Thus, we need to show that there is no torsion in the second Khovanov homology for other quantum gradings. This follows directly by combining Lemma~\ref{lem:UCT} and Theorem~\ref{thm:khovanov_positivity_general}.
\end{proof}

\begin{proof}[Proof of Theorem~\ref{thm:pos_fib} (torsion in second Khovanov homology)]
By Theorem~\ref{thm:khovanov_positivity}, the torsion subgroup of $Kh^{2,*}(L)$ vanishes exactly when $p_1(L)=0$. For a non-split positive link, Theorem~\ref{thm:burdened_fibered+positive} identifies this condition with fiberedness.
\end{proof}

\begin{rem}
    We remark that Lemma~\ref{lem:UCT} also shows that the potential infinitely many obstructions (one for every finitely presented abelian group $R$) from Theorem~\ref{thm:khovanov_positivity_general} are all implied by the integral first Khovanov homology and the torsion in integral second Khovanov homology.
\end{rem}

\subsection{Khovanov homology of split positive links}

By applying the Künneth formula for Khovanov homology of split
unions to our main result, we compute the Khovanov homology of split
positive links in homological degrees zero and one, as well as the
torsion in homological degree two.

\begin{theorem}\label{thm:split_positive}
Let
$L=L_1\sqcup\cdots\sqcup L_s$
be the split decomposition of a positive link into non-split positive
links. 
Then
\begin{align*}
    Kh^{0,-\chi(L)+2r}(L)=&\Z^{\binom{s}{r}}, \qquad 0\leq r\leq s,\\
    Kh^{1,2-\chi(L)+2r}(L)=&\Z^{p_1(L)\binom{s-1}{r}},\qquad 0\leq r\leq s-1,\\
    T\bigl(Kh^{2,4-\chi(L)+2r}(L)\bigr)=&
\Z_2^{p_1(L)\binom{s-1}{r}},
\qquad 0\leq r\leq s-1,  
\end{align*}
and all other Khovanov homology groups in homological degree zero and one
vanish
and $Kh^{2,j}(L)$ is torsion-free in every other quantum grading $j$.
\end{theorem}

\begin{proof}
For brevity, we write
\[
\chi=\sum_{a=1}^s\chi(L_a)
\qquad\text{and}\qquad
p=\sum_{a=1}^s p_1(L_a), 
\]
and $\chi_a=\chi(L_a)$ and $p_a=p_1(L_a)$.
By Theorem~\ref{thm:khovanov_positivity}, each non-split positive
factor $L_a$ satisfies
\begin{align}
Kh^{0,j}(L_a)
&=
\begin{cases}
\Z,&j=-\chi_a,\\
\Z,&j=2-\chi_a,\\
0,&\text{otherwise,}
\end{cases}
\label{eq:split-Kh0}\\
Kh^{1,j}(L_a)
&=
\begin{cases}
\Z^{p_a},&j=2-\chi_a,\\
0,&\text{otherwise,}
\end{cases}
\label{eq:split-Kh1}
\end{align}
and
\begin{equation}
T\bigl(Kh^{2,j}(L_a)\bigr)
=
\begin{cases}
\Z_2^{p_a},&j=4-\chi_a,\\
0,&\text{otherwise.}
\end{cases}
\label{eq:split-Kh2}
\end{equation}
For two links $J_1$ and $J_2$, the bigraded Künneth theorem
\cite[Corollary~12]{Khovanov_homology} gives a split short exact
sequence whose tensor-product terms in bidegree $(i,j)$ are
\[
Kh^{i_1,j_1}(J_1)\otimes Kh^{i_2,j_2}(J_2),
\qquad
i_1+i_2=i,\quad j_1+j_2=j,
\]
and whose $\operatorname{Tor}$-terms are
\[
\operatorname{Tor}\bigl(
Kh^{i_1,j_1}(J_1),Kh^{i_2,j_2}(J_2)
\bigr),
\qquad
i_1+i_2=i+1,\quad j_1+j_2=j.
\]
We apply this formula iteratively to
$L=L_1\sqcup\cdots\sqcup L_s$.
We first compute homological degree zero. In view of
\eqref{eq:split-Kh0}, from each factor $L_a$ we must choose one of two
generators, in quantum degree
$
-\chi_a$ or $2-\chi_a$.
Suppose that the second, or upper, quantum degree is chosen in exactly
$r$ of the $s$ factors. The resulting quantum degree is
\[
\sum_{a=1}^s(-\chi_a)+2r=-\chi+2r.
\]
There are exactly $\binom{s}{r}$ ways to choose these $r$ factors.
This proves
\[
Kh^{0,-\chi+2r}(L)=\Z^{\binom{s}{r}}.
\]

To obtain total homological degree one, we must choose a degree-one
class from exactly one factor, say $L_a$, and a degree-zero class from
each of the remaining $s-1$ factors. The degree-one group of $L_a$
has rank $p_a$ and quantum degree $2-\chi_a$. If the upper
degree-zero generator is chosen in exactly $r$ of the remaining
factors, then the total quantum degree is
\[
(2-\chi_a)
+\sum_{b\neq a}(-\chi_b)+2r
=2-\chi+2r.
\]
For a fixed $a$, there are $\binom{s-1}{r}$ such choices. Summing over
the distinguished factor $L_a$ gives total rank
\[
\sum_{a=1}^s p_a\binom{s-1}{r}
=p\binom{s-1}{r}.
\]
This proves the formula for $Kh^1(L)$.

Notice that all the groups in homological degrees zero and one are free abelian. Consequently, the
$\operatorname{Tor}$-terms in the iterated Künneth formula do not
contribute in homological degrees zero, one, or two. Indeed, a
Künneth $\operatorname{Tor}$-term contributing in homological degree
two involves homological degrees whose sum is three. Thus, at least
one of its two arguments lies in homological degree zero or one, where
the relevant groups are free.
The tensor-product terms in total homological degree two are of two
types:
\[
Kh^2(L_a)\otimes
\bigotimes_{b\neq a}Kh^0(L_b)
\]
for some $a$, and
\[
Kh^1(L_a)\otimes Kh^1(L_b)\otimes
\bigotimes_{c\neq a,b}Kh^0(L_c)
\]
for some $a<b$. The terms of the second type are free abelian and
therefore do not contribute any torsion. Hence all torsion in
$Kh^2(L)$ comes from the torsion subgroup of $Kh^2(L_a)$ in a single
factor and degree-zero generators in all remaining factors.
By \eqref{eq:split-Kh2}, the torsion subgroup contributed by the
factor $L_a$ is $\Z_2^{p_a}$ in quantum degree $4-\chi_a$.
Choosing the upper degree-zero generator in exactly $r$ of the
remaining $s-1$ factors gives total quantum degree
\[
(4-\chi_a)
+\sum_{b\neq a}(-\chi_b)+2r
=4-\chi+2r.
\]
For each $a$, there are $\binom{s-1}{r}$ such choices. Summing over
$a$ therefore gives
\[
T\bigl(Kh^{2,4-\chi+2r}(L)\bigr)
=
\Z_2^{
\left(\sum_{a=1}^s p_a\right)\binom{s-1}{r}}
=
\Z_2^{p\binom{s-1}{r}}.
\]
No other tensor-product or $\operatorname{Tor}$-term can contribute
torsion in homological degree two, which completes the proof.
\end{proof}

\subsection{Positive and fibered, but not braid positive}
In this section, we present an infinite family of fibered, positive knots that are not braid positive. 

\begin{proposition}\label{prop:infinite_family}
Let $L_n$ be the oriented link with diagram $D_n$ shown in Figure~\ref{fig:infinitefamily}, for $n\in \mathbb{N}$. Then $L_n$ is fibered and positive, but not braid positive for $n \geq 1$.
\end{proposition}

\begin{figure}[htbp]
\centering
\includegraphics[width = 12cm]{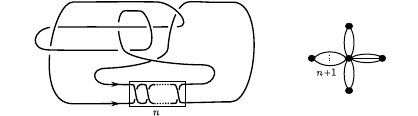}
\caption{The diagram $D_n$, where the box labeled $n$ denotes $n$ positive half-twists, and its Seifert graph $G_n= G_0(D_n)$. $L_2$ and $L_4$ correspond to the knots $K10n7$ and $K12n105$, respectively. $L_0$ is the torus knot $T(2,5)$, which is braid positive.}
\label{fig:infinitefamily}
\end{figure}

\begin{proof}
One checks that the diagram $D_n$ is positive and has reduced Seifert graph a tree. Therefore, Theorem~\ref{thm:burdened_fibered+positive} implies that $L_n$ is fibered. In order to show that $L_n$ is not braid positive, we show that Ito's normalized HOMFLYPT polynomial $\widehat{H}_{L_n}$ of $L_n$ contains negative coefficients. The statement follows from the result that $\widehat H_L$ is non-negative if $L$ is braid positive~\cite[Theorem~1.1]{Ito}.

First, observe that $L_n$ is a knot if $n$ is even; otherwise, it is a non-split link of two components (its linking number equals $(n+3)/2$). 
Applying Ito's skein relation~\cite[Equation~(1)]{Ito} to one of the $n$ crossings in the box leads to the following:
$$
    \begin{aligned}
     \widehat H_{L_n}(\alpha,z)=\begin{cases}
     \widehat H_{L_{n-2}}(\alpha,z)+\widehat H_{L_{n-1}}(\alpha,z) \, ,\, \, \quad \text{ if } n \mbox{ is even},\\
     \widehat H_{L_{n-2}}(\alpha,z)+z^2\, \widehat H_{L_{n-1}}(\alpha,z) \, ,\, \text{ if } n \mbox{ is odd}.
     \end{cases}
 \end{aligned}
$$

Set $\widehat H_{L_n}(\alpha, z) = \sum_{i\in\Z} P^{(i)}_{L_n}(\alpha) \cdot z^i$. We will prove that $P^{(0)}_{L_n}$, which we denote by $P_n$, is not positive. The surface $S_n$ obtained by applying the Seifert algorithm to $D_n$ realizes the $3$-genus of $L_n$ and $\chi(S_n)=\chi(G_n)=-n-3$, and therefore we get (via the HOMFLYPT function at SageMath and Ito's normalization) for the first links in the family:
\begin{align*}
    P_1 &=1-2\alpha^2-\alpha^3,  \\
    P_2&=4+2\alpha-2\alpha^2-\alpha^3.
 \end{align*}
 These computations can be accessed at~\cite{data}. Taking the coefficient of $z^0$ in Ito's skein relation above gives $P_n=P_{n-2}+P_{n-1}$ for even $n$ and $P_n=P_{n-2}$ for odd $n$. Induction therefore gives
\begin{align*}
     P_n=\begin{cases}
      (3+k)+2\alpha-2k\alpha^2-k\alpha^3, \quad & \text{ if } n= 2k,\\
     1-2\alpha^2-\alpha^3,  \quad \quad & \text{ if } n \text{ is odd},
     \end{cases}
 \end{align*}
 
\noindent which are non-positive polynomials for any $n>0$. 
 \end{proof}

\section{Khovanov homology of L-space knots and cables}\label{sec:comp}

In this section, we study the Khovanov homology of fibered, strongly quasipositive knots, of L-space knots, and of cable links. In particular, we show that Theorems~\ref{thm:khovanov_positivity} and~\ref{thm:pos_fib} do not extend to fibered, strongly quasipositive knots. On the other hand, we prove that Theorem~\ref{thm:pos_fib} extends to many infinite families of L-space knots, giving support for an affirmative answer to Conjecture~\ref{conj:Lspace}. 

Thurston classified knots into three disjoint classes: torus knots, satellite knots, and hyperbolic knots~\cite{Thurston}. Torus knots are fibered, and the (strongly) quasipositive torus knots are exactly the positive torus knots $T(p,q)$, $p/q>0$~\cite{Hedden}; these are also braid positive L-space knots~\cite{Moser}. Thus, for torus knots, these four positivity notions agree.

Satellite knots are those whose complement contains an incompressible, non-boundary parallel torus. Among them, a simple class is given by cable knots; we shift our attention to them in Section~\ref{subsection_cable}. In Section~\ref{subsection_hyperbolic} we focus on hyperbolic knots. 

\subsection{Cable links}\label{subsection_cable}

Given a knot $K$, its $(p,q)$-\textit{cable link} $K_{p,q}$ is the satellite link with pattern the torus link $T(p,q)$ in $S^1\times D^2$ and companion $K$. It is a knot precisely when $p$ and $q$ are coprime. Given a diagram $D$ of $K$, the \textit{standard diagram} $D_{p,q}$ of $K_{p,q}$ is obtained by taking $p$ parallel blackboard copies of $D$ and adding $n=pw-q$ negative $(1/p)$-twists to the $p$ parallel strands, where $w$ is the writhe of $D$ (see Figure~\ref{fig:trefoilcables} for an example). 
It is straightforward to check that the isotopy class of $K_{p,q}$ is independent of the chosen diagram. 
But we emphasize that the standard diagram $D_{p,q}$ depends on the diagram $D$ of $K$ we start with.
Since $K_{p,q}=K_{-p,-q}$ we can assume in the following that $p\geq2$. 

\begin{figure}[htbp]
\centering
\includegraphics[width=\textwidth]{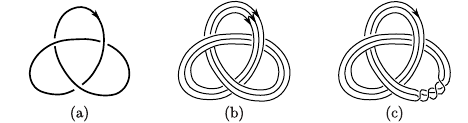}
\caption{A diagram $D$ of the positive trefoil knot is shown in~(a). Standard diagrams of its cables $D_{3,9}$ and $D_{3,5}$ are shown in (b) and (c), respectively.}
\label{fig:trefoilcables}
\end{figure}

If a knot $K$ is fibered, braid positive, positive, strongly quasipositive, or an L-space knot, then the following result summarizes conditions under which a cable of $K$ inherits these properties. 

\begin{theorem}\label{thm:properties_of_cabling}
Let $K$ be a knot, $p\geq2$, and $q\neq0$.
    \begin{enumerate}
        \item The cable link $K_{p,q}$ is fibered if and only if $K$ is fibered.
        \item Let $K$ be fibered and strongly quasipositive. Then the cable link $K_{p,q}$ is strongly quasipositive if and only if $q>0$.
        \item The cable $K_{p,q}$ is an L-space knot if and only if $K$ is an L-space knot, $p$ and $q$ are coprime, and $q \ge p\,(2g(K) - 1)$, where $g(K)$ denotes the $3$-genus of $K$.
        \item Let $K$ be a positive (braid positive) knot. Then the link $K_{p,q}$ is positive (braid positive) if $q\geq p \,w(K)$, where $w(K)$ denotes the minimal writhe of a positive (braid positive) diagram of $K$. 
    \end{enumerate}
\end{theorem}

\begin{proof}
$(1)$ is proven in~\cite{When_does_satellite_knot_fiber}. An explanation together with beautiful visualizations for one direction can be found at~\cite{BakerBlog}. $(2)$ follows from Proposition 2.4 and Theorem 2.11 in~\cite{Hedden_cabling_contact_structures_complex}, cf.~\cite[Corollary 4.14]{Hayden}. 
$(3)$ is proved by Hedden~\cite{Hedden_knot_floer_cablingII} and Hom~\cite{Hom_cabling_and_Lspace_surgeries}. 
$(4)$  Let $D$ be a positive (braid positive) diagram of $K$ with writhe $w(D)$. If $q\geq p \, w(D)$, then every crossing in the standard diagram $D_{p,q}$ is positive.
\end{proof} 

\begin{rem}\label{rem:not_pos}
Part (4) of Theorem~\ref{thm:properties_of_cabling} does not give any statements about the positivity (braid positivity) status of the cables of $K$ for $q<p \, w(K)$. In fact, it turns out that for a given positive (braid positive) knot $K$, the cables of $K$ with $0< q<p \, w(K)$ are often not positive (braid positive). This can be checked by computing standard obstructions for positivity and braid positivity. 

In~\cite{data} we show that all $(2,q)$-cables of $K=T(2,3)$, $T(2,5)$, $T(3,4)$, $T(3,5)$, $T(4,3)$, and $T(4,5)$ with $0<q<2\,w(K)$ have non-positive normalized HOMFLYPT polynomial and thus are not braid positive by~\cite{Ito}. A general obstruction is known for the slopes $q=\pm1$: Krishna~\cite{Krishna_TautFoliation_braidpositivity} proves that an $(n,\pm1)$-cable is braid positive only when its companion is the unknot. We do not know a complete braid positivity classification for the remaining slopes.

Sometimes the crossing number can also be used to obstruct positivity. If $D$ is a reduced, positive diagram with $c$ crossings of a fibered knot $K$ then it has minimal Seifert genus and $2g(K)-1=c-|s_0D|=c-s$. Since $K$ is fibered, the reduced Seifert graph $G_0(D)^{red}$ is a tree and hence has $s-1$ edges. Each of these edges represents at least two parallel edges of $G_0(D)$: otherwise, the corresponding crossing would be nugatory, contradicting the assumption that $D$ is reduced. Consequently, $
c\geq2(s-1)$. It follows that $4g(K)=2(c-s+1)=2c-2(s-1)\geq c$.
In particular, every positive, fibered knot satisfies
$c(K)\leq4g(K)$. This can also be deduced from~\cite[Corollary~5.1]{Cromwell}.

The $(2,3)$-cable of $T(2,3)$ is an L-space knot and hence fibered. It has Seifert genus $3$ and crossing number $15$~\cite[Table~8]{Du20}.  Therefore, $T(2,3)_{2,3}$ is not positive.
\end{rem}

Theorem~\ref{thm:properties_of_cabling} allows us to present examples of fibered, strongly quasipositive cable knots whose first Khovanov homology is nonzero.
The simplest such example we know is the $(2,1)$-cable of the $(2,3)$-torus knot, which was communicated to us by Lukas Lewark. 

\begin{theorem}\label{thm:verifying_conjecture}
For any $1\leq g\leq 200$
$$Kh^{1,j}\big(T(2,2g+1)_{2,1}\big)= \left\{ \begin{array}{lcc}
             \mathbb{Z} &    \mbox{ if } j = 3; \\
             0  &  \mbox{ otherwise} .
             \end{array}
   \right.$$

In particular, for $1 \leq g \leq 200$ the knot $T(2,2g+1)_{2,1}$ is strongly quasipositive, fibered, and its Khovanov homology in homological degree $1$ does not vanish. 
\end{theorem}

\begin{proof}
It follows from Theorem~\ref{thm:properties_of_cabling} (1) and (2) that $T(2,2g+1)_{2,q}$ is fibered and strongly quasipositive for any $g,q > 0$. 
We use KnotJob~\cite{KnotJob} and the KnotTheory package~\cite{knotatlas} for the computations of the Khovanov homologies. The computations can be accessed at~\cite{data}.
\end{proof}

The above computational results strongly suggest that Theorem~\ref{thm:verifying_conjecture} holds actually true for all $g\in\N$. However, we are not aware of general computations of the Khovanov homology of cables. Nevertheless, we now proceed with the proof of Theorem~\ref{thm:KHcable} showing that the Khovanov homology of certain cables of positive knots behaves like the Khovanov homology of a positive knot.

We will deduce Theorem~\ref{thm:KHcable} from a more general result about the family of twisted cable links, which we introduce now. 
Given $0 \leq m \leq p-1$, the {\it{twisted cable}} $K_{p,q;m}$ is the satellite link of $K$ with pattern given by the braid word\footnote{Here we represent a standard Artin generator or its inverse by $\pm i$.} 
$$w_{p,q}*[-1,-2,\ldots,-m],$$ where $w_{p,q}= [(1,2, \ldots, p-1)^q]$ is the standard $p$-stranded braid word of $T(p,q)$ in $S^1\times D^2$. As for genuine cables, we get standard diagrams $D_{p,q;m}$ of twisted cables, as shown in Figure \ref{fig:twistedcable}. 
It is clear that $K_{p,q;0} = K_{p,q} = K_{p,q+1;p-1}$. We write $K^{\circ}_{p,q;m}$ for any link obtained as the disjoint union of $K_{p,q;m}$ and a finite number (possibly zero) of unknots. 

\begin{figure}[htbp]
\centering
\includegraphics[width = 11cm]{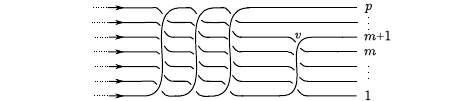}
\caption{Part of a diagram $D_{p,q;m}$ representing a twisted cable link. The diagram also contains $p$ blackboard copies of $D$ that are not depicted. In this example, $p=7$, $n=3$, and $m=4$, where $n=pw(D)-q$.}
\label{fig:twistedcable}
\end{figure}

The idea in proving Theorem~\ref{thm:KHcable} is to inductively reduce it to Theorem~\ref{thm:khovanov_positivity}. For that, we will need to know when a twisted cable is non-split. The following lemma gives us a criterion.

\begin{lemma}\label{lem:non_split}
    Let $K$ be a non-trivial knot and $p\neq 0$. Then the twisted cable $K_{p,q;m}$ of $K$ is non-split.
\end{lemma}

\begin{proof}
    Let $S$ be an embedded $2$-sphere not bounding a $3$-ball in $E_{K_{p,q;m}}=S^3\setminus \nu (K_{p,q;m})$. We want to show that one component of $E_{K_{p,q;m}}\setminus S$ contains no boundary component of $E_{K_{p,q;m}}$.
    Since $K_{p,q;m}$ is a satellite link, its exterior $E_{K_{p,q;m}}$ decomposes as 
    \begin{equation*}
        E_{K_{p,q;m}}=E_K\cup_T \big(V\setminus \nu(B_{p,q;m})\big),
    \end{equation*}
    where $E_K=S^3 \setminus \nu(K)$ denotes the exterior of $K$, $V\setminus \nu(B_{p,q;m})$ is the exterior of the braid
    \begin{align*}
    B_{p,q;m}=[1,2,\ldots,p-1]^q*[-1,-2,\ldots,-m]
\end{align*}
    in a solid torus $V$, and $T$ is the $2$-torus boundary of $E_K$ and $V$ along which we glue these spaces.
    By general position, we can assume that $S$ intersects $T$ transversely in a finite collection of disjoint circles $c_1,\ldots,c_n$. If $n=0$, then $S$ lies completely in $E_K$ or $S$ lies completely in $V\setminus \nu(B_{p,q;m})$. 
    Since $E_K$ is irreducible, $S$ cannot lie completely in $E_K$. 

    Next, we consider the case that $S$ lies completely in $V\setminus \nu(B_{p,q;m})$. First, we note that since $p\neq0$ any component of $B_{p,q;m}$ intersects every meridional disk of $V$. Since the solid torus $V$ is irreducible, every $2$-sphere $S$ in $V$ bounds a $3$-ball $B$ in $V$. In particular, there exists a meridional disk of $V$ that does not intersect $B$. This implies that no component of $B_{p,q;m}$ lies completely in $B$ and thus $S$ bounds a $3$-ball in $V\setminus \nu(B_{p,q;m})$.
    
    If $n>0$, we perform the following standard innermost circle argument to reduce $n$ by one and reduce this situation by induction to the case $n=0$. By the Schönflies theorem, there exists a circle, say $c_n$ that bounds a $2$-disk $D$ in $S\setminus (c_1\cup\cdots c_{n-1})$. If the disk $D$ lies completely in $E_K$, it has to be compressible since $K$ is not the unknot. Thus, after compressing any disk that lies completely in $E_K$, we can assume that $D$ is properly embedded in $V\setminus \nu(B_{p,q;m})$. But any disk properly embedded in a solid torus $V$ is either boundary parallel or a meridional disk. It cannot be the latter, since $p\neq 0$ would imply that $D$ intersects $B_{p,q;m}$. It follows that $c_n$ also bounds a disk $D'$ in $T$ such that there exist a $3$-ball $B$ in $V\setminus \nu(B_{p,q;m})$ whose boundary is $D\cup_{c_n} D'$. Since $p\neq 0$, no component of $B_{p,q;m}$ lies in $B$. Thus we can isotope $D\subset S$ over $B$ into $E_K$. This removes the intersection curve $c_n$ and yields an isotopic $2$-sphere $S'$ that intersects $T$ in $n-1$ curves. 
\end{proof}

Next, we state and prove the generalization of Theorem~\ref{thm:KHcable} for twisted cables, from which we will deduce later Theorem~\ref{thm:KHcable}.

\begin{theorem}\label{thm:twisted_cable}
	Let $K$ be a positive knot. Then, for $q-1 \geq p \geq 2$, and $1\leq m \leq p-1$ we have 
	\begin{align*}
		Kh^{i,*}(K_{p,q;m})&=0 \,\textrm{ for every }\,  i<0, \,\textrm{ and }\,  \\
		Kh^{0,j}(K_{p,q;m})&=\begin{cases}
			\Z\,&\textrm{ if }\,  j=1-p\chi(K)+q(p-1)-m\pm1, \\
			0 \,&\textrm{else,}
		\end{cases}\\
		Kh^{1,j}(K_{p,q;m})&=\begin{cases}
			\Z^{p_1(K)}\,&\textrm{ if }\,  j=2-p\chi(K)+q(p-1)-m, \\
			0 \,&\textrm{else,}
		\end{cases}\\
        T\big(Kh^{2,j}(K_{p,q;m})\big)&=\begin{cases}
			\Z_2^{p_1(K)}\,&\textrm{ if }\,  j=4-p\chi(K)+q(p-1)-m, \\
			0 \,&\textrm{else.}
		\end{cases}
	\end{align*}
\end{theorem}

\begin{proof}
Let $D$ be a positive diagram of $K$ with writhe $w$. First, we recall that $D_{p,q}$ for $q\geq pw$ is a positive diagram. By analyzing its Seifert smoothing, we see that every circle in $s_0D$ yields $p$ circles in $s_0D_{p,q}$. A simple analysis on the $0$-chords shows that $G_0(D_{p,q})^{red}$ has the same cyclomatic number as $G_0(D)^{red}$ and thus $p_1(K_{p,q})=p_1(K)$ for all $q\geq pw$.

Now, we write $n=pw-q$ and consider the lexicographic order on the triples $(p,n,m)$ and proceed by induction. 
We start with the base case $p=2$. Then $m=1$ and
$
K_{2,q;1}=K_{2,q-1}$.
Consider the standard diagram $D_{2,q-1}$ of $K_{2,q-1}$. It has $2w-q+1$ negative half-twists. Performing a $0$-smoothing at any crossing in the negative twisting region leads to a diagram of the unknot with writhe
$
w_0=2w-q$,
while performing a $1$-smoothing at the same crossing leads to the standard diagram $D_{2,q}$ of $K_{2,q}=K_{2,q+1;1}$. The integer appearing in Corollary~\ref{cor:Bsmoothing} is
$$
\frac{w(D_{2,q-1})-w_0-1}{2}=q-1.
$$
Thus Corollary~\ref{cor:Bsmoothing} gives
$$
Kh^{i,j}(K_{2,q;1})=Kh^{i,j+1}(K_{2,q+1;1})
$$
whenever $i$ is not equal to $q-1$ or $q$. Since $q\geq3$, this isomorphism holds for every $i\leq1$. Iterating it, we eventually reach
$$
K_{2,2w+1;1}=K_{2,2w},
$$
which has a positive standard diagram with $p_1(K_{2,2w})=p_1(K)$. The total quantum-grading shift is $2w-q+1$. Since
$
-\chi(K_{2,2w})=-2\chi(K)+2w$,
the formulas for $Kh^0$ and $Kh^1$, including their quantum gradings, follow from Theorem~\ref{thm:khovanov_positivity} and~\cite[Proposition~8]{Patterns_in_Khovanov_homology}.
It remains to consider torsion in homological degree two. If $q>3$, then $2$ is not one of the exceptional homological degrees $q-1$ and $q$, so the same isomorphism holds in degree two. If $q=3$, the exact sequence in Corollary~\ref{cor:Bsmoothing} shows that the corresponding map from $Kh^2(K_{2,4;1})$ to $Kh^2(K_{2,3;1})$ is injective and that its cokernel embeds into the integral Khovanov homology of the unknot. This cokernel is therefore torsion-free, so the map induces an isomorphism between the torsion subgroups. Consequently, the formula for the torsion in $Kh^2$ also follows by induction.

Now let $p\geq3$. For the below induction, we also include the auxiliary base triple $(p,0,0)$. This triple represents the ordinary cable $K_{p,pw}$, which has a positive standard diagram. Moreover, Schubert's cabling formula for the genus~\cite{Schubert} says that
$$
-\chi(K_{p,pw})=-p\chi(K)+pw(p-1).
$$
Thus, the required formulas for this base triple follow from Theorem~\ref{thm:khovanov_positivity} and~\cite[Proposition~8]{Patterns_in_Khovanov_homology}.

For the induction step, we assume that the statement is true for all $(p',n',m')<(p,n,m)$ with $p\geq 3$. Consider the diagram $D_{p,p w-n;m}$ with writhe 
$$\overline{w}=p^2w-n(p-1)-m.$$ Let $v$ be the highest negative crossing in the last partial twist (see Figure~\ref{fig:twistedcable}). Then the $1$-smoothing $D_1$ at $v$ is the standard diagram of $K_{p,pw-n;m-1}$. If $m>1$, this corresponds to the smaller triple $(p,n,m-1)$. If $m=1$ and $n>0$, we use
$$
K_{p,pw-n;0}=K_{p,pw-(n-1);p-1},
$$
which corresponds to the smaller triple $(p,n-1,p-1)$. Finally, if $m=1$ and $n=0$, then $D_1$ represents the positive base cable $K_{p,pw}$.

\begin{claim}
The $0$-smoothing of $D_{p,pw-n;m}$ at $v$ leads to a diagram $D_0$ with writhe $w_0$ satisfying the following conditions (that we will prove in Lemmas~\ref{lem:writhe_comp} and~\ref{lem:families} below):
\begin{itemize}
    \item[(1)] $w_0 - \overline{w} + 3 <0$.
    \item[(2)] $D_0$ is either a (non-positive) diagram of a positive link or a (non-standard) diagram of $K^{\circ}_{p',p'w-n',m'}$ with $(p',n',m')<(p,n,m)$.
    \end{itemize}
\end{claim}

We can then apply Lemma~\ref{lem:les} to see that the Khovanov homology groups of these links fit into the following exact sequence
\begin{equation*}
  Kh^{i'-1,j'}(D_0) \rightarrow  Kh^{i,j+1}(K_{p,p  w-n;m-1})\rightarrow Kh^{i,j}(K_{p,p  w-n;m})\rightarrow Kh^{i',j'}(D_0)
\end{equation*}
with $i'= (w_0-\overline{w}+1)/2 +i$. 
The first condition in the claim implies that $i' < 0$ if $i\leq1$. If $D_0$ represents a positive link, then we get from~\cite[Proposition~8]{Patterns_in_Khovanov_homology} that $Kh^{i',j'}(D_0)$ is vanishing. Otherwise, the second condition in the claim together with Lemma~\ref{lem:kunneth} allows us to apply the inductive hypothesis to deduce that $Kh^{i',j'}(D_0)$ is vanishing. Analogously, we conclude that $Kh^{i'-1,j'}(D_0)$ is trivial. 
Thus we conclude that $Kh^{i,j}(K_{p,p  w-n;m})=Kh^{i,j+1}(K_{p,p  w-n;m-1})$ and the theorem follows by induction since $K_{p,pw}$ is positive. 

It remains to prove the formula for torsion in homological degree two. We set $i=2$ in the same exact sequence. Condition~(1) implies that $i'\leq0$. If $i'<0$, both outside terms vanish, and we obtain
$$
Kh^{2,j}(K_{p,pw-n;m})=Kh^{2,j+1}(D_1).
$$
If $i'=0$, the exact sequence contains
$$
0\longrightarrow Kh^{2,j+1}(D_1)
\longrightarrow Kh^{2,j}(K_{p,pw-n;m})
\longrightarrow Kh^{0,j'}(D_0).
$$
The group $Kh^{0,j'}(D_0)$ is free abelian. Indeed, this is known for positive links, and in the other case it follows from the induction hypothesis and Lemma~\ref{lem:kunneth}. Therefore the quotient of $Kh^{2,j}(K_{p,pw-n;m})$ by the image of $Kh^{2,j+1}(D_1)$ is torsion-free. It follows that
$
T\bigl(Kh^{2,j}(K_{p,pw-n;m})\bigr)
=
T\bigl(Kh^{2,j+1}(D_1)\bigr)
$
as claimed.
\end{proof}

\begin{lemma}\label{lem:writhe_comp}
Assume that $p\geq3$ and $q\geq p+1$. If we orient the diagram $D_0$ in the proof of Theorem~\ref{thm:twisted_cable} in such a way that the orientations of the components not involved in $v$ are preserved, then
$$
w(D_0)-w(D_{p,pw-n;m})+3<0.
$$
\end{lemma}

\begin{proof}
Figure~\ref{fig:twistedcable0} shows the $0$-smoothing $D_0$ of a standard diagram $D_{p,pw-n;m}$ of a twisted cable. We add the dotted line $\mathcal{L}$ to simplify diagrammatic descriptions. 

\begin{figure}[htbp]
\centering
\includegraphics[width = 10.5cm]{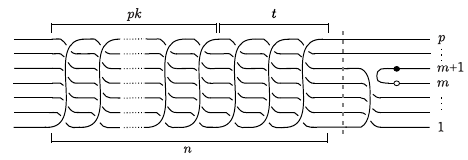}
\caption{Part of diagram $D_0$ together with the dotted grey line $\mathcal{L}$ illustrating the proofs of Lemmas~\ref{lem:writhe_comp} and~\ref{lem:families}.}
\label{fig:twistedcable0}
\end{figure}

Write $n=pk+t$, with $0\leq t<p$.
We set 
\begin{itemize}
    \item $x$: number of strands crossing $\mathcal{L}$ oriented from right to left.
    \item $x_m$: number of strands crossing $\mathcal{L}$ oriented from right to left in position $i$, with $i\in \{2, \dots m\}$. The labeling of position starts from below, as in Figure~\ref{fig:twistedcable0}.
    \item $t_d$: number of strands crossing $\mathcal{L}$ at position $i$ oriented from right to left, with $p-t+1 \leq i \leq p$. 
    \item $u$: it takes the value $1$ (resp. $-1$) if the situation is as in Figure~\ref{fig:casesu}(a) (resp.~Figure~\ref{fig:casesu}(b)).
\end{itemize}

\begin{figure}[htbp]
\centering
\includegraphics[width = 10.5cm]{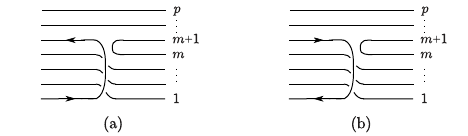}
\caption{The two possible situations determining the value of the parameter $u$: if $D_0$ is oriented as in (a) then $u=1$; in case (b) then $u=-1$.}
\label{fig:casesu}
\end{figure}

We compute the writhe $w_0$ of $D_0$ in terms of the above parameters. To do so, we split its crossings into four classes (I -- IV) and compute their writhes ($w_{\operatorname{I}}$ -- $w_{\operatorname{IV}}$) separately, so that $w_0= w_{{\operatorname{I}}} + w_{{\operatorname{II}}} + w_{{\operatorname{III}}} + w_{{\operatorname{IV}}}$.  

\noindent - \underline{I}: contains those crossings coming from the positive diagram $D$ (i.e.~not created by any of the twists). Each crossing in $D$ contributes to $w_{\operatorname{I}}$ with $((p-x)-x)^2$. Hence, the total contribution is $$w_{\operatorname{I}}=w (p-2x)^2.$$ 

The next two classes contain those crossings created by the $n=pk+t$ twists. \\

\noindent - \underline{II}: contains those crossings involved in the $pk$ twists which are farther from $\mathcal{L}$. These twists can be subdivided into $k$ blocks, each of them containing $p$ consecutive twists (i.e.~each block contains exactly one full twist). Fix a strand and compute the sign of the crossings created when it wraps around the other strands: if it crosses $\mathcal{L}$ from right to left (resp.~from left to right), then it creates $(p-x)$ positive and $(x-1)$ negative crossings (resp.~$x$ positive and $(p-x-1)$ negative crossings). Hence, we get  
\begin{align*}
        w_{{\operatorname{II}}} &= k \Big(x \big((p-x)-(x-1)\big) + (p-x) \big(x-(p-x-1)\big)\Big) \\ & = (n-t) \big(1-p+4x-4x^2/p\big)\\
        &= k(p-p^2+4xp-4x^2).
\end{align*}

\noindent - \underline{III}: contains those crossings involved in the $t$ twists which are closest to the dotted line. A similar reasoning as before leads to \begin{align*}w_{{\operatorname{III}}} &= t_d \big((p-x)-(x-1)\big) + (t-t_d) \big(x-(p-x-1)\big) \\& = 2t_d (p-2x) + t (2x-p+1).\end{align*}

\noindent - \underline{IV}: contains those crossings created by the partial twist related to $m$. There are two possible situations depicted in Figure~\ref{fig:casesu}, from which we read off that $$w_{\operatorname{IV}} = u (2x_m-m+1).$$

The writhe of the twisted cable $D_{p,pw-n;m}$ is computed to be $$w\left(D_{p,pw-n;m}\right) = \overline{w} = p^2 w - n(p-1) - m.$$

Now, it follows that

\begin{align*}
    \beta &\coloneqq w(D_0) - \overline{w} + 3 \\
&= w(p - 2x)^2 + k(p - p^2 + 4xp - 4x^2) + 2t_d(p - 2x) + t(2x - p + 1) \\
&\quad + u(2x_m - m + 1) - p^2w + n(p - 1) + m + 3 \\
&= (wp^2 + 4wx^2 - 4wxp) + (-4x^2k - p^2k + 4xpk + pk) + 2t_d(p - 2x) \\
&\quad +(2xt - pt + t) + u(2x_m - m + 1) - p^2w + (kp^2 - kp + tp - t) + m + 3 \\
&= 4x(k - w)(p - x) + 2xt + 2t_d(p - 2x) + u(2x_m - m + 1) + m + 3.
\end{align*} 

Next, we bound some of the terms in $\beta$. 
First, observe from Figure~\ref{fig:twistedcable0} that
$
1\leq x\leq p-1$.
Since $q\geq p+1$, we have
$$
pk+t=n=pw-q\leq pw-p-1.
$$
As $t\geq0$, this implies
$
k\leq w-2$.
Consequently,
$$
4x(k-w)(p-x)\leq -8x(p-x).
$$
Moreover, since $0\leq x_m\leq m-1$ and $u\in\{-1,1\}$, we have
$$
u(2x_m-m+1)\leq m-1.
$$
Together with $m\leq p-1$, this gives
$$
u(2x_m-m+1)+m+3\leq2p.
$$
Therefore,
$$
\beta
\leq
-8x(p-x)+2xt+2t_d(p-2x)+2p.
$$
Next, we distinguish two cases.

\medskip

\noindent\textbf{Case 1: $p-2x>0$.}
Since $t_d$ counts some of the $x$ strands crossing $\mathcal L$ from right to left, we have $t_d\leq x$. Using also $t\leq p-1$, we obtain
$$
\begin{aligned}
\beta
&\leq -8x(p-x)+2x(p-1)+2x(p-2x)+2p\\
&=-4xp+4x^2-2x+2p\\
&=2(p-x)(1-2x).
\end{aligned}
$$
Since $1\leq x<p$, we have $p-x>0$ and $1-2x<0$. Hence $\beta<0$.

\medskip

\noindent\textbf{Case 2: $p-2x\leq0$.}
We define
$
i=p-t$.
Since $0\leq t\leq p-1$, we have $1\leq i\leq p$. The number $t_d$ counts the strands oriented from right to left among the last $t$ positions. There are exactly $i=p-t$ positions that are not counted by $t_d$. Therefore, at most $i$ of the $x$ strands oriented from right to left are not counted by $t_d$, and hence
$$
t_d\geq x-i.
$$
Because $p-2x\leq0$, this yields
$$
2t_d(p-2x)\leq2(x-i)(p-2x).
$$
Using $t=p-i$, we obtain
$$
\begin{aligned}
\beta
&\leq -8x(p-x)+2x(p-i)+2(x-i)(p-2x)+2p\\
&=-4xp+4x^2+2ix-2ip+2p.
\end{aligned}
$$
Since $x\leq p-1$, we have $4x^2\leq4x(p-1)$, and therefore
$$
\begin{aligned}
\beta
&\leq -4xp+4x(p-1)+2ix-2ip+2p\\
&=2(p-2x)+2i(x-p).
\end{aligned}
$$
The first summand is non-positive because $p-2x\leq0$. The second summand is strictly negative because $i\geq1$ and $x-p\leq-1$. Hence $\beta<0$.
\end{proof}

Recall that the diagram $D_0$ in the proof of Theorem~\ref{thm:twisted_cable} is obtained from a standard diagram $D_{p,pw-n;m}$ of a twisted cable by smoothing the negative crossing~$v$. 

\begin{lemma}\label{lem:families}
Assume that $q-1\geq p$. There exists exactly one component $C$ in
$D_0$ that is involved in the smoothing at $v$. If we preserve the
orientation of all other components, then there exists a choice of
orientation on $C$ such that $D_0$ is either a possibly non-positive
diagram of a positive link or a non-standard diagram of
$K^{\circ}_{p',p'w-n';m'}$, where
\[
(p',n',m')<(p,n,m)
\qquad\text{and}\qquad
p'w-n'\geq p'+1.
\]
\end{lemma}

\begin{proof}
The two arcs with vertical tangencies in Figure~\ref{fig:twistedcable0} are joined by the smoothing, so they lie on a unique component $C$. Orient $C$ as in Figure~\ref{fig:cases6}. Push the marked arc to the right through the diagram until its white endpoint meets $\mathcal L$ at position $a$, with positions read modulo $p$. This produces the six configurations in Figure~\ref{fig:cases6}.

\begin{figure}[htbp]
\centering
\includegraphics[width = 11cm]{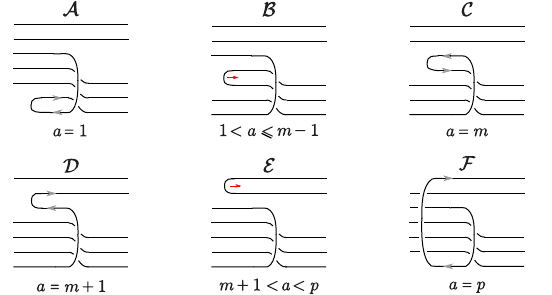}
\caption{The six configurations in the proof of Lemma~\ref{lem:families}, determined by the endpoint position $a$. The original orientation of $D_{p,q;m}$ is from left to right.}
\label{fig:cases6}
\end{figure}

In cases $\mathcal B$ and $\mathcal E$, continue the same push until one of $\mathcal A,\mathcal C,\mathcal D,\mathcal F$ occurs. Cases $\mathcal C$ and $\mathcal D$ directly delete a pair of oppositely oriented parallel strands and give a twisted cable with $p'=p-2$. In case $\mathcal A$, the moves in Figure~\ref{fig:specialcase} slide the block of crossings through the companion diagram and put the result into standard twisted-cable form. Here, a \emph{negative twist} means a crossing in the twisting block whose sign is negative with the chosen orientation. If none remain, the diagram is positive; otherwise, these moves collect all of them into the standard negative twisting block. The red box may have to pass through the companion diagram more than once before reaching panel~(d).

\begin{figure}[htbp]
\centering
\includegraphics[width = 4.5cm]{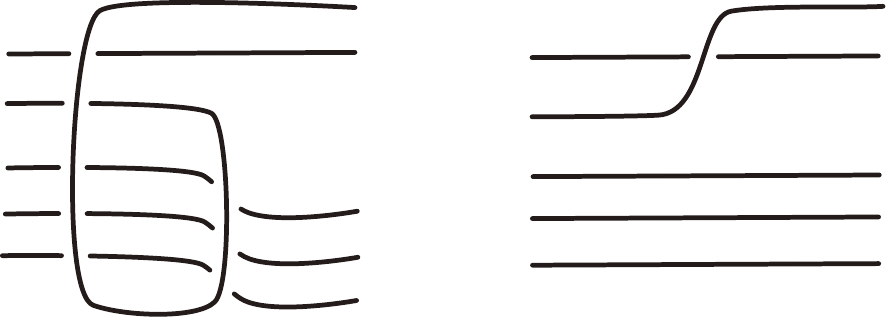}
\caption{The isotopies used in the proof of Lemma~\ref{lem:families} in case~$\mathcal F$.}
\label{fig:specialcase_F}
\end{figure}

\begin{figure}[htbp]
\centering
\includegraphics[width = 11.3cm]{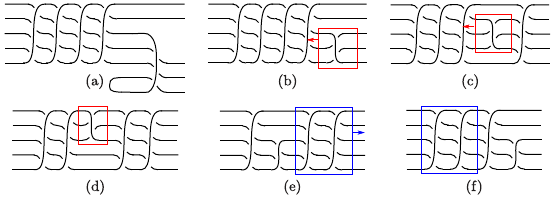}
\caption{The isotopies used in the proof of Lemma~\ref{lem:families} in case~$\mathcal A$. From (b) to (d), the red box is pushed to the left. Starting at (e), the blue box is pushed to the right through the diagram to obtain (f).}
\label{fig:specialcase}
\end{figure}

In case $\mathcal F$, first perform the isotopy in Figure~\ref{fig:specialcase_F}. A half-turn then gives standard form when the remaining twisting block is positive; if negative crossings remain, slide their box to the right exactly as in case $\mathcal A$. Thus, all four terminal cases give either a positive diagram or a twisted cable with fewer strands. When $m=1$ or $m=p-1$, the same local moves may separate an additional unknot; this accounts for the superscript $\circ$ and does not affect the induction by Lemma~\ref{lem:kunneth}.

It remains to check the slope condition. Write
\[
n=kp+t,\qquad n'=k'p'+t',
\qquad 0\leq t\leq p-1,\quad 0\leq t'\leq p'-1.
\]
Thus, $k$ and $k'$ are the numbers of full negative twists in the two
diagrams.



\begin{figure}[htbp]
  \centering
  \setkeys{Gin}{width=.4\linewidth} 
    \subfloat[]{\includegraphics{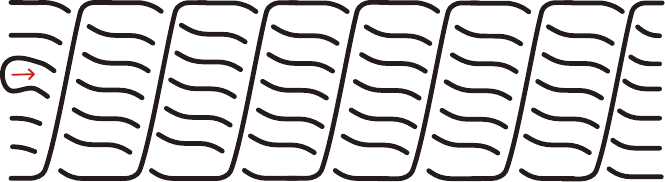}\label{fig:a}}%
    \hfil
    \subfloat[]{\includegraphics{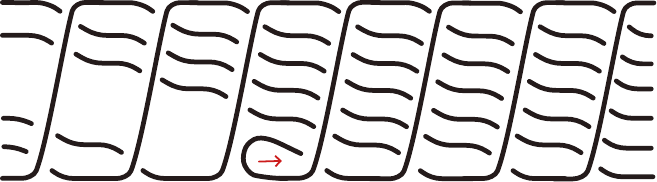}\label{fig:b}}\\
    \subfloat[]{\includegraphics{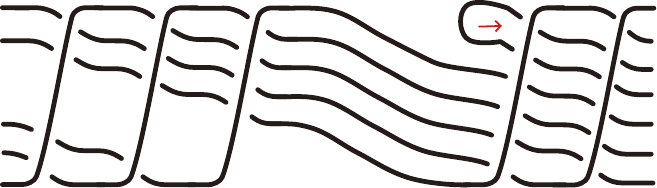}\label{fig:c}}%
    \hfil
    \subfloat[]{\includegraphics{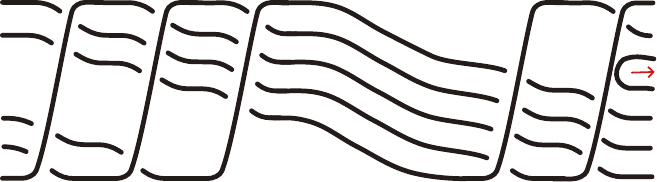}\label{fig:d}}%
 \caption{Arc pushed across a full negative twist.}%
 \label{fig:fulltwist_reduction}%
\end{figure}

We claim that $k'=k$. We adopt the convention that whenever the resulting diagram $D'$ is an actual cable, we identify $D_{p,q}$ with $D_{p,q+1;p-1}$. We use this convention throughout the proof, and in particular in the definition of $k'$. 

Consider pushing the marked arc through one full negative twist. The two strands are removed from that full twist, while the remaining strands form exactly one full negative twist, as illustrated in Figure~\ref{fig:fulltwist_reduction}. Observe that whenever the marked arc passes through a full negative twist, the remainder $t$ can decrease by $1$ or $2$, or remain unchanged, depending on its value relative to $p$.

We now analyze the cases individually. In Cases $\mathcal{B}$ and $\mathcal{E}$, we simply continue pushing the marked arc. In Cases $\mathcal{C}$, $\mathcal{D}$, and $\mathcal{F}$, the isotopies only modify the final incomplete twist and therefore leave the number of full twists unchanged.

It remains to consider Case $\mathcal{A}$, where a slight subtlety arises. The moves in Figure~\ref{fig:specialcase} could potentially affect a full twist if $t$ becomes zero immediately before the moves are performed. We now show that this does not change the number of full twists. If pushing the marked arc results in $t=0$, then immediately before the removal step we must have $t=1$ or $t=2$. These are exactly the configurations shown in Figure~\ref{fig:specialcases3}, and in each case the resulting configuration falls under one of the cases already analyzed. Moreover, these exhaust all possibilities: any configuration in which the value of $k$ could change must pass through one of these configurations first. Thus, the number of full twists is preserved throughout, and consequently $k'=k$.

\begin{figure}[htbp]
\centering
\includegraphics[width = 11.3cm]{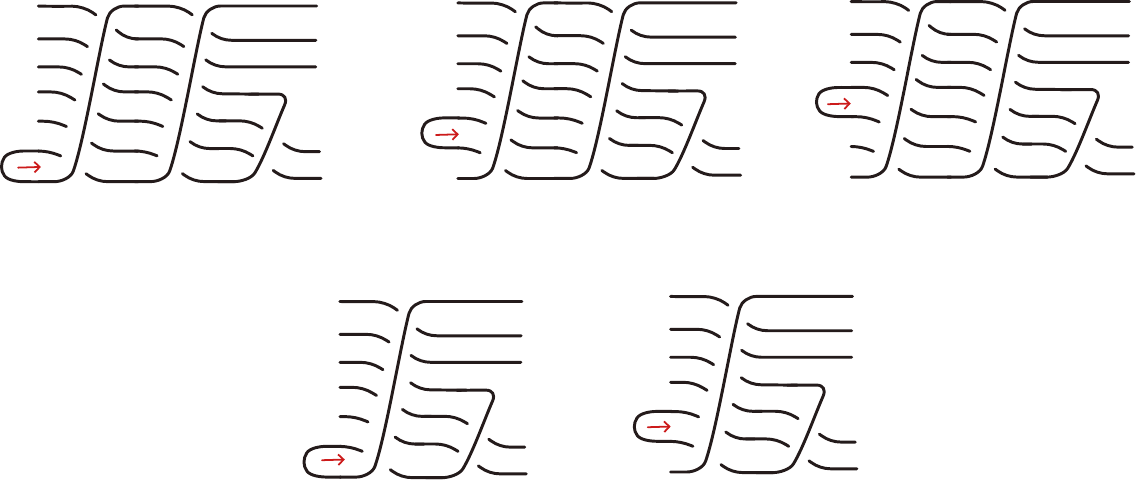}
\caption{The configurations with $t=1$ and $t=2$.}
\label{fig:specialcases3}
\end{figure}

Now we have
\[
p+1\leq q=pw-n=pw-(kp+t)=p(w-k)-t, 
\]
and since $0\leq t\leq p-1$, it follows that
$w-k\geq2$. 
Combining the latter with $k'=k$ and $0\leq t'\leq p'-1$ we obtain the inequality
\begin{align*}
q' = p'(w-k')-t' = p'(w-k)-t' \geq 2p'-(p'-1)=p'+1,
\end{align*}
as claimed.
\end{proof}

\begin{proof}[Proof of Theorem~\ref{thm:KHcable}]
We recall that Schubert’s cabling formula~\cite{Schubert} for a cable says that
\[
-\chi(K_{p,q}) = -p\chi(K) + q(p - 1).
\]
First, we consider the case that $q \geq pw$. Then the cable $K_{p,q}$ is positive by Theorem~\ref{thm:properties_of_cabling}, and the result follows directly from Theorem~\ref{thm:khovanov_positivity}, together with the above cabling formula.

For $q<pw$, we use
$K_{p,q}=K_{p,q+1;p-1}$.
Since $q\geq p$, we have $(q+1)-1\geq p$, and therefore Theorem~\ref{thm:twisted_cable} applies to $K_{p,q+1;p-1}$. Moreover,
$$
1-p\chi(K)+(q+1)(p-1)-(p-1)
=
1-\chi(K_{p,q}).
$$
Thus, the quantum gradings in Theorem~\ref{thm:twisted_cable} agree with those in the statement of Theorem~\ref{thm:KHcable}. This proves the result, including the boundary case $q=p$.
\end{proof}

\begin{rem}
In the proof of Theorem~\ref{thm:twisted_cable}, we require the $(p,q)$-cables to satisfy $q-1\geq p \geq 2$ so we can achieve the inequality for $\beta$ in Lemma~\ref{lem:writhe_comp}. 
However, computational experiments show that for \(p,w\le 30\) and \(q\ge3\)
the inequality \(\beta<0\) still holds (see~\cite{data}).
This suggests that the writhe estimate in the previous lemma holds in a
substantially wider range. It would be interesting to determine whether
Theorem~\ref{thm:KHcable} itself remains true for all \(q\ge3\).
\end{rem}

\begin{proof}[Proof of Corollary~\ref{cor:fibered_cable}]
By Theorem~\ref{thm:KHcable}, the group $Kh^{1,*}(K_{p,q})$ vanishes exactly when $p_1(K)=0$. The same condition characterizes torsion-freeness of $Kh^{2,*}(K_{p,q})$. Theorem~\ref{thm:burdened_fibered+positive} identifies $p_1(K)=0$ with fiberedness of the positive knot $K$.
\end{proof}

\begin{rem}
We remark that Lemma~\ref{lem:UCT} also gives versions of Theorem~\ref{thm:KHcable} and Corollary~\ref{cor:fibered_cable} for coefficients in a general finitely presented abelian group $R$.
\end{rem}

In further search for a counterexample to Conjecture~\ref{conj:Lspace}, it is natural to consider cables of more general L-space knots. Possible candidates to be counterexamples are iterated cables of L-space knots. 

\begin{theorem}\label{thm:iteraded_cable}
Let $K$ be the $(2,3)$-cable of $T(2,3)$. Then, for every $q\geq2$, the Khovanov homology groups of $K_{2,q}$ with $\mathbb{Z}_2$-coefficients satisfy
\begin{align*}
    Kh^{i,j}(K_{2,q};\Z_2)= \begin{cases}
        \Z_2 &\textrm{ if } (i,j)=(0,1-\chi(K_{2,q})\pm1), \\
        0 &\textrm{ for all other groups with }\, i\leq1.
        \end{cases}
\end{align*}
 In particular, any $2$-stranded cable of $T(2,3)_{2,3}$ that is an L-space knot has Khovanov homology with $\Z_2$-coefficients of the above form.
\end{theorem} 

To prove Theorem~\ref{thm:iteraded_cable}, we need the following result for $2$-cables.  

\begin{theorem}\label{thm:general_two_cables}
		Let $K$ be a knot. Then, for every coefficient group $R$ and every $q\geq2$, the first Khovanov homologies of its $2$-cables are related by
	\begin{equation*}
		Kh^{i,j}(K_{2,q};R)=Kh^{i,j+2-q}(K_{2,2};R) \,\,\textrm{ for all }\,\,i\leq 1.
	\end{equation*}
\end{theorem}

\begin{proof}
Let $D$ be a knot diagram of $K$ with writhe $w$. In our convention, $n=2w-q$ is the signed number of negative half-twists in $D_{2,q}$, and this diagram has writhe $4w-n$. Thus, if $n<0$, the twisting region consists of $-n=q-2w$ positive half-twists. If $n>0$, let $v$ be a crossing in the negative twisting region. If $n\leq0$, introduce a cancelling positive--negative pair next to that region and take $v$ to be its negative crossing. In both cases, the $0$-smoothing is an unknot diagram with writhe $w_0=n-1$, while the $1$-smoothing is $D_{2,q+1}$. Thus Corollary~\ref{cor:Bsmoothing} applies and gives
\begin{align*}
Kh^{i,j}(K_{2,q};R)=Kh^{i,j+1}(K_{2,q+1};R)
\end{align*}
for all $i\neq q,q+1$. Since $q\geq2$, this holds for $i\leq1$; iteration proves the theorem.
\end{proof}

\begin{proof}[Proof of Theorem~\ref{thm:iteraded_cable}]
Since $K$ is not positive (see Remark~\ref{rem:not_pos}), Theorem~\ref{thm:KHcable} does not apply here. In~\cite{data} we also show that $K_{2,q}$ is not braid positive for small values of $q$ (and conjecturally it is not positive for all $q>0$), so we cannot apply Theorem~\ref{thm:khovanov_positivity} here. 

By Theorem~\ref{thm:general_two_cables}, it is enough to compute the first Khovanov homology of one $2$-stranded cable of $K$. The integral and rational computations did not terminate on our machines, so we make only the stated $\Z_2$-coefficient claim. Using~\cite{knotatlas}, we computed the required base cable and obtained the displayed groups; the data can be accessed at~\cite{data}.

Since $T(2,3)_{2,3}$ has Seifert genus $3$, Theorem~\ref{thm:properties_of_cabling}(3) implies that if $K_{2,q}$ is an L-space knot, then $q>10$, and the second claim follows.
\end{proof}

\subsection{Hyperbolic knots}\label{subsection_hyperbolic}
We found hyperbolic knots that are fibered and strongly quasipositive whose first Khovanov homology is non-vanishing. The idea to construct these examples was to start with a quasipositive Seifert surface of a fibered strongly quasipositive knot and plumb a Hopf band to it (recall that Hopf plumbing preserves fiberedness and strongly quasipositivity), so that the new resulting knot is hyperbolic. The following example was found by a computer search motivated by the above idea.

\begin{ex}\label{ex:hyp} Let $K$ be the hyperbolic knot given by the closure of the $4$-braid
\begin{equation*}
    [(3, 2, 2, 1), -2, (3, 2, 2, 1)^2, 1,1, -2,-2,-2,-2],
\end{equation*}
shown in Figure~\ref{fig:hyp_braid}. 
From the knot Floer chain complex of $K$ (computed via Szabó's program~\cite{KnotFloerhomologycalculator}), we deduce that $K$ is fibered and its $3$-genus agrees with $\tau(K)$. Therefore, $K$ is strongly quasipositive by~\cite{Hedden}.
However, $Kh^{1,*}(K) \neq 0$, as shown in Table~\ref{fig:Kh_cbbaBcbbacbbaaaBBBB}.
\end{ex}

\begin{figure}[htbp] 
 	\centering
 \includegraphics[width=\textwidth]{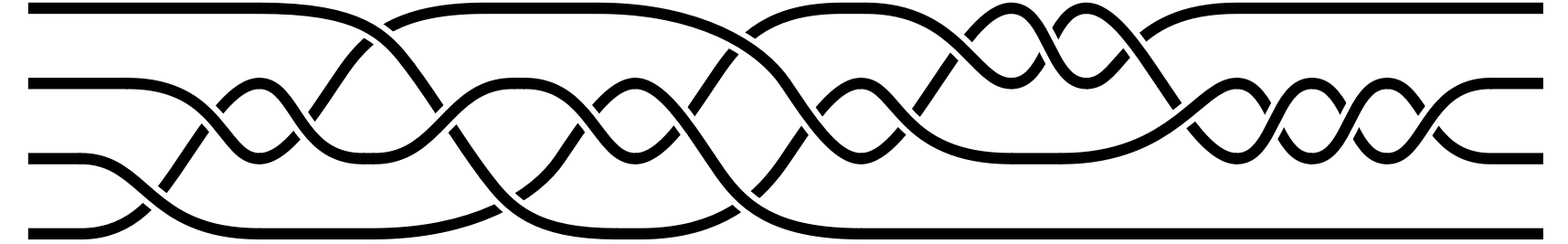}
 	\caption{The braid from Example~\ref{ex:hyp}, picture created via KLO~\cite{KLO}. 
}
  \label{fig:hyp_braid} 
 \end{figure}

  \begin{figure}[htbp] 
 	\centering
 	\includegraphics[width=\textwidth]{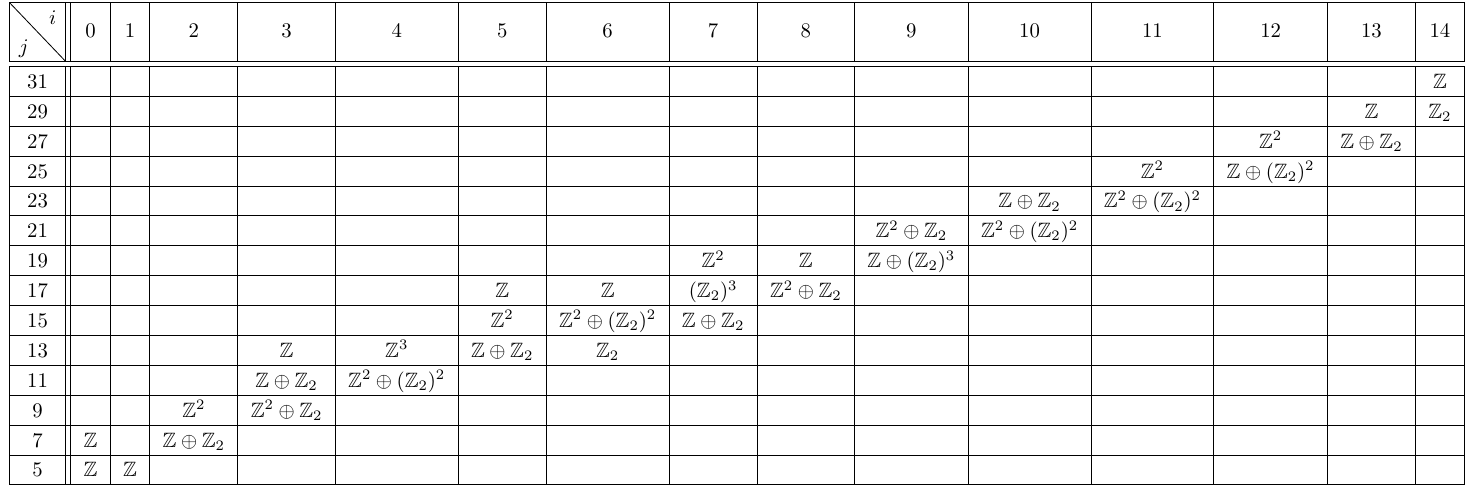}
 	\caption{Khovanov homology of the hyperbolic knot given as the closure of the braid in Figure~\ref{fig:hyp_braid}.}
  \label{fig:Kh_cbbaBcbbacbbaaaBBBB} 
 \end{figure}

For L-space knots, the situation is more complicated. Currently, all hyperbolic L-space knots that are not known to be braid positive are contained in the infinite families from~\cite{Himeno} and~\cite{BakerKegel}. For the first two non-braid positive, hyperbolic L-space knots from~\cite{Himeno}, we have computed the Khovanov homology. However, the positivity obstruction is satisfied. For the other knots from~\cite{Himeno}, the computations did not terminate. In the following, we will demonstrate that the braid positivity obstructions are also satisfied for the infinite family $\{K_n\}_{{n>0}}$ from~\cite{BakerKegel}, where the knot $K_n$ is given as the closure of the $4$-braid
 \begin{align*}
     \beta_n= [(2, 1, 3, 2)^{2n+1}, -1, 2, 1, 1, 2].
 \end{align*}
  In~\cite{BakerKegel} it is shown that all $K_n$ are hyperbolic L-space knots. Since the normalized HOMFLYPT polynomial of $K_1$ is not positive, it follows that $K_1$ is not braid positive. When $n>1$, the braid positivity status of $K_n$ is currently unclear. The following theorem shows that the Khovanov homology groups of $K_n$ fulfill the Khovanov homology positivity obstructions and thus the positivity status of $K_n$ remains open. 

 \begin{theorem}\label{thm:Baker_Kegel_family}
     For every $n\geq 1$, the Khovanov homology of $K_n$ satisfies
     \begin{align*}
    Kh^{i,j}(K_n)&= \begin{cases}
        \Z &\textrm{ if } (i,j)=(0,4+8n\pm1)=(0,2g(K_n)\pm1), \\
        0 &\textrm{ for all other groups with }\, i\leq1,
        \end{cases}\\
        T\bigl(Kh^{2,*}(K_n)\bigr)&=0.
\end{align*}
 \end{theorem}

The Khovanov homology table of $K_1$ is shown in Figure~\ref{fig:K1}.

\begin{proof}
Let $D_n$ be the diagram obtained as the closure of $\beta_n$, whose writhe is $w_n= 8n+7$. The $0$-smoothing of the unique negative crossing $v$ in $D_n$ produces a diagram of the unknot with writhe $w_0=0$, while the $1$-smoothing of $v$ produces a braid positive diagram $(D_n)_1$, and thus its first Khovanov homology vanishes and its second Khovanov homology has no torsion. Next, we apply Corollary~\ref{cor:Bsmoothing}, where $u=\frac{w_n-w_0-1}{2}=4n+3 \geq 7$, and get
    \begin{align*}
        Kh^{i,j}(K_n)=Kh^{i,j+1}\big((D_n)_1\big) \,\,\textrm{ for all } \,\,i\leq6.
    \end{align*}
We then obtain the result for $i = 1,2$ using Theorem~\ref{thm:pos_fib} together with the observation that $\chi((D_n)_1) = -4 - 8n$, and for $i < 1$ by applying the results of Section~\ref{sec:other_obstructions}.
\end{proof}

 \begin{figure}[htbp] 
 	\centering
 	\includegraphics[width=\textwidth]{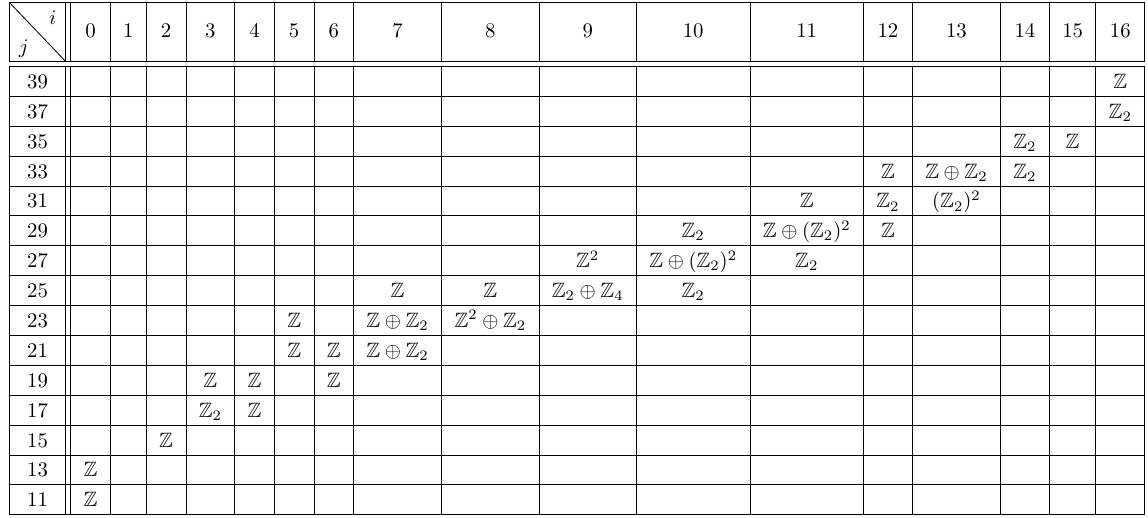}
 	\caption{Khovanov homology of the hyperbolic L-space knot~$K_1$.}
  \label{fig:K1} 
 \end{figure}

\let\MRhref\undefined
\bibliographystyle{hamsalpha}
\bibliography{ref.bib}

\end{document}